\documentclass[oneside,english]{amsart}
\usepackage[T1]{fontenc}
\usepackage[latin9]{inputenc}
\usepackage{color}
\usepackage{babel}
\usepackage{refstyle}
\usepackage{amstext}
\usepackage{amsthm}
\usepackage{amssymb}
\usepackage{stmaryrd}
\usepackage[unicode=true,pdfusetitle,
 bookmarks=true,bookmarksnumbered=false,bookmarksopen=false,
 breaklinks=false,pdfborder={0 0 1},backref=false,colorlinks=true]
 {hyperref}
\hypersetup{
 linkcolor=red, urlcolor=red, citecolor=red}

\makeatletter

\AtBeginDocument{\providecommand\propref[1]{\ref{prop:#1}}}
\AtBeginDocument{\providecommand\lemref[1]{\ref{lem:#1}}}
\AtBeginDocument{\providecommand\corref[1]{\ref{cor:#1}}}
\RS@ifundefined{subsecref}
  {\newref{subsec}{name = \RSsectxt}}
  {}
\RS@ifundefined{thmref}
  {\def\RSthmtxt{theorem~}\newref{thm}{name = \RSthmtxt}}
  {}
\RS@ifundefined{lemref}
  {\def\RSlemtxt{lemma~}\newref{lem}{name = \RSlemtxt}}
  {}

\numberwithin{equation}{section}
\numberwithin{figure}{section}
\theoremstyle{definition}
\newtheorem{defn}{\protect\definitionname}[section]
\theoremstyle{plain}
\newtheorem{thm}{\protect\theoremname}[section]
\theoremstyle{plain}
\newtheorem{prop}{\protect\propositionname}[section]
\theoremstyle{remark}
\newtheorem{rem}{\protect\remarkname}[section]
\theoremstyle{plain}
\newtheorem{lem}{\protect\lemmaname}[section]
\theoremstyle{plain}
\newtheorem{cor}{\protect\corollaryname}[section]

\makeatother

\providecommand{\corollaryname}{Corollary}
\providecommand{\definitionname}{Definition}
\providecommand{\lemmaname}{Lemma}
\providecommand{\propositionname}{Proposition}
\providecommand{\remarkname}{Remark}
\providecommand{\theoremname}{Theorem}

\begin{document}
\title[Strichartz estimates w/o loss outside many convex obstacles]{Strichartz estimates without loss outside many strictly convex obstacles}
\author{David Lafontaine {*}}
\thanks{{*} d.lafontaine@bath.ac.uk, University of Bath, Department of Mathematical
Sciences}
\begin{abstract}
We prove Strichartz estimates without loss for Schrödinger
and wave equations outside finitely many strictly convex obstacles
verifying Ikawa's condition, introduced in \cite{Ikawa2}. We extend
the approach introduced in \cite{Schreodinger,Waves} for the two
convex case.
\end{abstract}

\maketitle

\section{Introduction}

Let $(M,g)$ be a Riemannian manifold of dimension $d$. We are interested
in the Schrödinger
\begin{align}
\begin{cases}
i\partial_{t}u-\Delta_{g}u=0\\
u(0)=u_{0}
\end{cases}\label{eq:lw-1}
\end{align}
 and wave equations on $M$
\begin{align}
\begin{cases}
\partial_{t}^{2}u-\Delta_{g}u=0\\
(u(0),\partial_{t}u(0))=(f,g),
\end{cases}\label{eq:lw}
\end{align}
where $\Delta_{g}$ is the Laplace-Beltrami operator. A key to
study the perturbative theory and the nonlinear problems associated
with these equations is to understand the size and the decay of the
linear flows. One tool to quantify these decays is the so-called \textit{Strichartz
estimates
\begin{gather*}
\Vert u\Vert_{L^{q}(0,T)L^{r}(M)}\leq C_{T}\left(\Vert u_{0}\Vert_{\dot{H}^{s}}+\Vert u_{1}\Vert_{\dot{H}^{s-1}}\right),\ \text{(Waves)}\\
\Vert u\Vert_{L^{q}(0,T)L^{r}(M)}\leq C_{T}\Vert u_{0}\Vert_{L^{2}},\ \text{(Schrödinger)}
\end{gather*}
}where $(p,q)$ has to follow an admissiblity condition given by
the scaling of the equation, respectively
\begin{gather*}
\frac{1}{q}+\frac{d}{r}=\frac{d}{2}-s,\ \frac{1}{q}+\frac{d-1}{2r}\leq\frac{d-1}{4},\\
\frac{2}{q}+\frac{d}{r}=\frac{d}{2},\ (q,r,d)\neq(2,\infty,2),
\end{gather*}
for the Schrödinger and wave equations.

These estimates have a long story, beginning with the work of \cite{Strichartz}
for the $p=q$ case in $\mathbb{R}^{n}$, extended to all exponents
by \cite{GV85}, \cite{LindbladSogge}, and \cite{KeelTao}. For the
wave equation in a manifold without boundary, the finite speed of
propagation shows that it suffices to work in local coordinates to
obtain local Strichartz estimates. This path was followed by \cite{Kapitanskii},
\cite{MoSeSo}, \cite{SmithC11}, and \cite{Tataruns}. The case of
a manifold with boundary, where reflexions have to be dealt with,
is more difficult. Estimates outside one convex obstacle for the wave
equation were obtained by \cite{SS95}, following the parametrix construction
of Melrose and Taylor, which gives an explicit representation of the
solution near diffractive points, and for the Schrödinger equation
later by \cite{MR2672795}.

The first local estimates for the wave equation on a general domain
were shown by \cite{BLP} for certain ranges of indices, then extended
by \cite{BSS}. These estimates cannot be as good as in the flat case
: \cite{OanaCounterex} showed indeed that a loss has to occur if
some concavity is met, because of the formation of caustics. Recently,
\cite{ILPAnnals} and \cite{ILLPGeneral} obtained almost sharp local
Strichartz estimates inside a convex domain. 

One obstruction to the establishment of global estimates without loss
is the presence of trapped geodesics. Under a non trapping assumption, such estimates
were established for the wave equation by the works of \cite{SmithSoggeNonTrapping},
\cite{MR2001179} and \cite{Metcalfe}. For the Schrödinger flow in
the boundaryless case, \cite{BoucletTzvetkov}, \cite{Bouclet},
\cite{HassellTaoWunsch}, \cite{StaffTata} obtained the estimates
in several non-trapping geometries.

When trapped geodesics are met, \cite{MR2066943} showed that a loss
with respect to the flat case has to occur for the wave equation
in the global $L^{2}$ integrability of the flow, and his counterpart,
the smoothing estimate, for the Schrödinger equation, which write
respectively in the flat case as
\begin{gather*}
\Vert(\chi u,\chi\partial_{t}u)\Vert_{L^{2}(\mathbb{R},\dot{H}^{s}\times\dot{H}^{s-1})}\lesssim\Vert u_{0}\Vert_{\dot{H}^{s}}+\Vert u_{1}\Vert_{\dot{H}^{s-1}}\ \text{(Waves)},\\
\Vert\chi u\Vert_{L^{2}(\mathbb{R},H^{1/2})}\lesssim\Vert u_{0}\Vert_{L^{2}\ }\text{(Schrödinger)}.
\end{gather*}

Despite this loss in the smoothing estimate, \cite{MR2720226} showed
Strichartz estimates without loss for the Schrödinger equation in
an asymptotically euclidian manifold without boundary for which the
trapped set is sufficiently small and exhibits an hyperbolic dynamic. 

Following this breakthrough, we recently proved in \cite{Schreodinger,Waves}
global Strichartz estimates without loss for Schrödinger and wave
equations outside two strictly convex obstacles, exhibiting in the
boundary case the first trapped situation where no loss occurs. The
goal of this paper is to extend this result to the case of the exterior
of $N\geq3$ convex obstacles, which is in many aspects a counterpart
with boundaries of the framework studied in \cite{MR2720226} .

In this $N$-convex obstacles setting, there is infinitely many trapped
rays. Therefore, there is a competition between the large number of
parts of the flow that remain trapped between the obstacles and the
decay of each such part. For a sufficient decay to hold, this competition
has to occur in a favorable way. This is the so called Ikawa condition:
\begin{defn}[Ikawa condition, 1: strong hyperbolicity]
There exists $\alpha>0$ such that the following condition holds
\begin{equation}
\sum_{\gamma\in\mathcal{P}}\lambda_{\gamma}d_{\gamma}e^{\alpha d_{\gamma}}<\infty.\label{eq:IK1}
\end{equation}
\end{defn}
Here $\mathcal{P}$ denotes the set of all primitive periodic trajectories,
$d_{\gamma}$ the length of the trajectory $\gamma$ and $\lambda_{\gamma}=\sqrt{\mu_{\gamma}\mu'_{\gamma}}$,
where $\mu_{\gamma}$ and $\mu'_{\gamma}$ are the two eigenvalues
of modulus smaller than one of the Poincaré map associated with $\gamma$.
This condition was first introduced by \cite{IkawaMult} when investigating
the decay of the local energy of the wave equation. Notice that it is in particular
automatically verified when the obstacles are sufficiently far from each other. It
is the analog of the topologic pressure condition arising in \cite{MR2720226}.

We will moreover suppose the second part of the Ikawa condition to
be verified, namely, denoting by $\Theta_{i}$ the obstacles:
\begin{defn}[Ikawa condition, 2: no obstacle in shadow]
For all $i,j,k$ pairwise distincts,
\begin{equation}
\text{Conv}(\Theta_{i}\cup\Theta_{j})\cap\Theta_{k}=\emptyset.\label{eq:IK2}
\end{equation}

At the difference of the first one, and excepting the degenerated
situation where a periodic trajectory is tangent to an obstacle, this
condition may be purely technical (it permits to construct solutions
without been preoccupied by the shadows induced by the obstacles) and
should be avoided with a more careful analysis. 

We are now in position to state our result.
\end{defn}
\begin{thm}
\label{th}Let $(\Theta_{i})_{1\leq i\leq N}$ be a finite family
of smooth strictly convex subsets of $\mathbb{R}^{3}$, such that Ikawa's
conditions (\ref{eq:IK1}) and (\ref{eq:IK2}) hold, and $\Omega=\mathbb{R}^{3}\backslash\underset{1\leq i\leq N}{\cup}\Theta_{i}$
. Then, under the non-endpoint admissibility conditions:
\begin{gather*}
\frac{1}{q}+\frac{3}{r}=\frac{3}{2}-s,\ \frac{1}{q}+\frac{1}{r}\leq\frac{1}{2},\ q\neq\infty,\ \text{(Waves)}\\
\frac{2}{q}+\frac{3}{r}=\frac{3}{2},\ (q,r)\neq(2,6),\ \text{(Schrödinger)}
\end{gather*}
global Strichartz estimates without loss hold for both Schrödinger
and wave equations in $\Omega$ :
\begin{gather*}
\Vert u\Vert_{L^{q}(\mathbb{R},L^{r}(\Omega))}\lesssim\Vert u_{0}\Vert_{\dot{H}^{s}}+\Vert u_{1}\Vert_{\dot{H}^{s-1}},\ \text{(Waves)}\\
\Vert u\Vert_{L^{q}(\mathbb{R},L^{r}(\Omega))}\lesssim\Vert u_{0}\Vert_{L^{2}}.\ \text{(Schrödinger)}
\end{gather*}
\end{thm}

\subsection*{Overview of the proof}

We generalise the approach introduced in \cite{Waves,Schreodinger}. 

As we dealt with the Schrödinger equation outside two convex obstacles
in \cite{Schreodinger} and showed in \cite{Waves} how to adapt the
work to the wave equation, the main novelty of this note is how to
handle the $N$-convex framework, and therefore we present a detailed
proof of our main result in the more intricate case of the Schrödinger
equation, and briefly explain how to adapt it to the wave equation
with the material of \cite{Waves} in the last section.

In the flat case, the smoothing estimate permits to stack Strichartz
estimates in time $\sim h$ for data of frequency $\sim h^{-1}$ to
show global estimates. As remarked in \cite{MR2720226}, the logarithmic
loss that appears in our setting in the smoothing estimate can be compensated if 
we show Strichartz estimates in time $h|\log h|$
instead of $h$ near the trapped set, provided a smoothing estimate
without loss in the non trapping region is at hand. Therefore, our
first section is devoted to prove such an estimate, using a commutator
argument together with the escape function construction of Morawetz,
Raltson and Strauss \cite{MorRS}. We then show that we can reduce
ourselves to data micro-locally supported near trapped trajectories,
and that remain in a neighbourhood of it in logarithmic times. We extend
to the $N$-convex framework the construction of an approximate solution
for such data done in \cite{Waves,Schreodinger} following ideas of
\cite{Ikawa2,IkawaMult} and \cite{MR1254820}, and finally, we show
that under the strong hyperbolicity assumption (\ref{eq:IK1}),
this construction gives a sufficient decay.

\subsection*{Notations}
\begin{itemize}
\item We denote by $\mathcal{K}\subset T^{\star}\Omega\cup T^{\star}\partial\Omega$
the trapped set, which is is composed of infinitely many periodic
trajectories,
\item and by $\mathcal{P}$ the set of all primitive periodic trajectories,
that is, followed only once, 
\item the operator $\psi(-h^{2}\Delta)$ localizes at frequencies $|\xi|\in[\alpha_{0}h^{-1},\beta_{0}h^{-1}],$
we refer to \cite{MR2672795} for the definition of this operator,
\item the set $\mathcal{I}$ is the set of all stories of reflexions, that
is all finites sequences $(j_{1},\cdots,j_{k})$ with values in $\llbracket1,\cdots,N\rrbracket$
such that $j_{i}\neq j_{i+1}$,
\item moreover, we will adopt all the notations introduced in \cite{Schreodinger}.
Let us in particular recall that 
\[
\Phi_{t}:T^{\star}\Omega\cup T^{\star}\partial\Omega\longrightarrow T^{\star}\Omega\cup T^{\star}\partial\Omega
\]
 denotes the billiard flow on $\Omega$: $\Phi_{t}(x,\xi)$ is the
point attained after a time $t$ from the point $x$ in the direction
$\frac{\xi}{|\xi|}$ at the speed $|\xi|$, following the laws of
geometrical optics,
\item finally, let us recall that the spatial and directional components
of $\Phi_{t}$ are respectively denoted $X_{t}$ and $\Xi_{t}$. 
\end{itemize}

\section{Smoothing effect without loss outside the trapped set}

Let us recall the smoothing effect with logarithmic loss obtained
in \cite{MR2066943} in our framework of a family of strictly convex
obstacle verifying Ikawa's condition:
\begin{prop}
\label{prop:smooth_logloss}For any $\chi\in C_{c}^{\infty}(\mathbb{R}^{3})$
and any $u_{0}\in L^{2}(\Omega)$ such that $u_{0}=\psi(-h^{2}\Delta)u_{0}$,
we have
\begin{equation}
\Vert\chi e^{it\Delta_{D}}u_{0}\Vert_{L^{2}(\mathbb{R},L^{2})}\lesssim(h|\log h|)^{\frac{1}{2}}\Vert u_{0}\Vert_{L^{2}}.\label{eq:smooth_logloss}
\end{equation}
\end{prop}
The aim of this first section is to prove a smoothing effect without
loss outside the trapped set:
\begin{prop}[Local smoothing without loss in the non trapping region]
\label{prop:smooth_wo}Let $\phi\in C_{c}^{\infty}(\mathbb{R}^{3}\times\mathbb{R}^{3})$
be supported in the complementary of the trapped set, $\mathcal{K}^{c}$.
Then we have, for $u_{0}=\psi(-h^{2}\Delta)u_{0}$
\begin{equation}
\Vert\text{Op}_{h}(\phi)e^{-it\Delta_{D}}u_{0}\Vert_{L^{2}(\mathbb{R},H^{1/2}(\Omega))}\lesssim\Vert u_{0}\Vert_{L^{2}}.\label{eq:lswlnt}
\end{equation}
\end{prop}
\begin{proof}
We will use the same strategy as in \cite{MR2720226} lemma 2.2, adapting
the proof in the case of a domain with boundary. Notice that, for
any operator $A$,
\begin{equation}
\langle Au,u\rangle(T)-\langle Au,u\rangle(0)=\int_{0}^{T}\int_{\Omega}\langle[i\Delta,A]u,u\rangle+\int_{0}^{T}\int_{\partial\Omega}\langle Au,\partial_{n}u\rangle.\label{eq:smooth_com}
\end{equation}
Thus, if we find an operator $A$ of order $0$ such that $[i\Delta,A]$
is elliptic and positive on the support of $\phi$ and such that the
border term
\[
B=\int_{0}^{T}\int_{\partial\Omega}\langle Au,\partial_{n}u\rangle d\sigma dt
\]
is essentially positive, we shall obtain the desired estimate. 

\subsection*{Notations}

If $b\in C^{\infty}(\mathbb{R}^{3}\times\mathbb{R}^{3})$ is a real
symbol such that $b\geq0$, and $b\geq\alpha$ on $U$, we use Garding
inequality on symbols of the form
\[
b-\alpha\frac{a\bar{a}}{\left(\sup|a|\right)^{2}}\geq0,
\]
where $a\in C_{c}^{\infty}(\mathbb{R}^{3}\times\mathbb{R}^{3})$ is
supported in $U$. Notice that we have $\text{Op}_{h}(a\bar{a})=\text{Op}_{h}(a)\text{Op}_{h}(a)^{\star}+O(h)$.

Moreover, we will denote, in this section and this section only, $\Phi$
for the operator associated to $\phi$.

\subsection*{The symbol of A at the border as an operator acting on Schrödinger
waves}

We perform the semi-classical time change of variable to write:
\[
B=h^{-1}\int_{0}^{hT}\int_{\partial\Omega}\langle A(e^{iht}u_{0}),\partial_{n}(e^{iht}u_{0})\rangle d\sigma dt
\]

We use the strategy of \cite{MorRS} to derive the symbol of $A$
at the border as an operator acting on Schrödinger waves. Let us consider
$A$ as an operator acting on $\partial\Omega\times\mathbb{R}$. Notice
that, because $\partial\Omega\times\mathbb{R}$ is nowhere characteristic
for the semi-classical Schrödinger flow, there exists an operator
$Q$ of order zero such that for any semi-classical Schrödinger wave
$v$
\begin{equation}
Av_{|\partial\Omega\times\mathbb{R}}=Q(\partial_{n}v).\label{eq:smooth_pn}
\end{equation}
Let $q$ be the symbol of this operator. Let $(x_{0},t_{0})\in\partial\Omega\times\mathbb{R}$
and $(\eta,\tau)\in T_{(x_{0},t_{0})}(\partial\Omega\times\mathbb{R})$.
We denote by $\psi_{\pm}$ the two distinct solutions of the Eikonal
equations
\begin{gather*}
|\nabla\psi(x)|^{2}=-\tau,\\
\psi_{\pm}(x)=x\cdot\eta\text{ on }\partial\mathcal{O},
\end{gather*}
that are well defined in a neighborhood of $x_{0}$ as soon as $\tau-\eta^{2}>0$:
indeed, extending $n$ in a small neighborhood of the border, one
can always take
\[
\psi_{\pm}=x\cdot\eta\pm\sqrt{\tau-\eta^{2}}n.
\]
For $\lambda>0$, consider, extending $\partial_{n}\psi_{\pm}$ in
a neighborhood of $x_{0}$ in $\Omega$
\[
v_{\lambda}=\frac{e^{i\lambda(\psi_{+}+t\tau)}-e^{i\lambda(\psi_{-}+t\tau)}}{(i\lambda)(\partial_{n}\psi_{+}-\partial_{n}\psi_{-})},
\]
which is solution of an approximate semi-classical Schrödinger equation
\begin{flalign*}
i\partial_{t}v_{\lambda}-\lambda^{-1}\Delta v_{\lambda} & =O(\lambda^{-1}),\\
v_{\lambda} & =0\ \text{\text{on }\ensuremath{\partial}\ensuremath{\mathcal{O}}}.
\end{flalign*}
verifying, in a neighborhood of $x_{0}$ in $\partial\mathcal{O}$
\[
\partial_{n}v_{\lambda}=e^{i\lambda(x\cdot\eta+t\tau)}.
\]
But, the principal symbol of $Q$ can be computed as
\begin{multline*}
q(x_{0},t_{0},\eta,\tau)=\lim_{\lambda\rightarrow\infty}e^{-i\lambda(x_{0}\cdot\eta+t_{0}\cdot\tau)}Q(e^{i\lambda(x\cdot\eta+t\tau)})(x_{0},t_{0})\\
=\lim_{\lambda\rightarrow\infty}e^{-i\lambda(x_{0}\cdot\eta+t_{0}\cdot\tau)}Q(\partial_{n}v_{\lambda})(x_{0},t_{0}).
\end{multline*}
By the Duhamel formula, the difference between $v_{\lambda}$ and
the solution of the actual equation $w_{\lambda}$ is bounded in a
neighborhood of $(x_{0},t_{0})$ by
\[
|w_{\lambda}-v_{\lambda}|\lesssim\lambda^{-1},
\]
therefore, we can replace $v_{\lambda}$ by $w_{\lambda}$, which
is an exact Schrödinger wave, in the limit and make use of (\ref{eq:smooth_pn})
to get:
\begin{gather*}
q(x_{0},t_{0},\eta,\tau)=\lim_{\lambda\rightarrow\infty}e^{-i\lambda(x_{0}\cdot\eta+t_{0}\cdot\tau)}Q(\partial_{n}w_{\lambda})(x_{0},t_{0})\\
=\lim_{\lambda\rightarrow\infty}e^{-i\lambda(x_{0}\cdot\eta+t_{0}\cdot\tau)}A(w_{\lambda})(x_{0},t_{0})=\lim_{\lambda\rightarrow\infty}e^{-i\lambda(x_{0}\cdot\eta+t_{0}\cdot\tau)}A(v_{\lambda})(x_{0},t_{0})\\
=\left(\frac{a(x,d\psi_{+})-a(x,d\psi_{-})}{2(\partial_{n}\psi_{+}-\partial_{n}\psi_{-})}\right)(x_{0},t_{0}).\\
\end{gather*}
And we conclude that, this computation been valid for $\tau-\eta^{2}>0$
\begin{gather}
q(x_{0},t_{0},\eta,\tau)=\left(\frac{a(x_{0},\xi_{+})-a(x_{0},\xi_{-})}{(\xi_{+}-\xi_{-})\cdot n(x_{0})}\right),\label{eq:smoothq1}\\
\xi_{\pm}=\eta\pm\sqrt{\tau-\eta^{2}}n(x_{0}).\label{eq:smoothq2}
\end{gather}
Notice that $\xi_{\pm}$ is a pair of reflected rays.

\subsection*{The escape function}

Let $(y,\eta)\notin\mathcal{K}$. The generalized broken ray starting
from $(y,\eta)$ is composed of a finite number of segments, thus,
the construction of \cite{MorRS}, Section 5, holds to construct a
ray function starting from $(y,\eta)$, that is, a function $p_{0}$
satisfying
\begin{gather*}
\xi\cdot\nabla p_{0}(x,\xi)\geq0,\ \frac{p_{0}(x,\xi)-p_{0}(x,\xi')}{(\xi-\xi^{'})\cdot n(x)}\geq0,
\end{gather*}
and
\[
\eta\cdot\nabla p_{0}(y,\eta)>0,\ \frac{p_{0}(y,\eta)-p_{0}(y,\eta')}{(\eta-\eta')\cdot n(y)}>0.
\]
 Therefore, by compactness, we can construct a function $a$ such
that
\begin{gather}
\xi\cdot\nabla a(x,\xi)\geq0,\frac{a(x,\xi)-a(x,\xi^{'})}{(\xi-\xi^{'})\cdot n(x)}\geq0\label{eq:smooth-a0-1}\\
\xi\cdot\nabla a(x,\xi)>0,\ \frac{a(x,\xi)-a(x,\xi^{'})}{(\xi-\xi^{'})\cdot n(x)}>0,\text{ on \ensuremath{V\supset\supset} supp \ensuremath{\phi}}.\label{eq:smooth-a0-2}
\end{gather}
Finally, notice that, because the construction of \cite{MorRS} follows
the rays and because the trapped set is invariant by the flow, we
can construct $a$ in such a way that
\begin{equation}
a=0\text{ near }\mathcal{K}.\label{eq:0-out}
\end{equation}
Remark that, as in \cite{MorRS}, Section 1, such an $a$ can be approximated
by a polynomial in order to justify the above integration by parts.

\subsection*{A first estimate}

Let $\delta>0$. Because of (\ref{eq:smooth-a0-1}), (\ref{eq:smoothq1}),
$q$ is real-valued and positive on $\{\tau-\eta^{2}\geq0\}$, therefore,
there exists $\epsilon>0$ small enough so that, on $\{\tau-\eta^{2}\geq-\epsilon\}$
we have, with the notations of (\ref{eq:smoothq1})
\begin{equation}
\Re e\frac{a(x_{0},\xi_{+})-a(x_{0},\xi_{-})}{(\xi_{+}-\xi_{-})\cdot n(x_{0})}\geq-\delta/2.\label{eq:a_bord_ok}
\end{equation}
and, for $|\alpha|\leq2(d+1)=8$
\begin{equation}
\big|\Im m\ \partial_{x,t,\xi,\tau}^{\alpha}\frac{a(x_{0},\xi_{+})-a(x_{0},\xi_{-})}{(\xi_{+}-\xi_{-})\cdot n(x_{0})}\big|\leq\delta/2\label{eq:a_bord_ok_IM}
\end{equation}
Now, let $\chi$ be positive and supported in $\{\tau-\eta^{2}\geq-2\epsilon\}$
and such that $\chi=1$ in $\{\tau-\eta^{2}\geq-\epsilon\}.$ We decompose
$a$ as the sum
\[
a=\chi a+(1-\chi)a.
\]
Note that $(1-\chi)a$ is supported away from the characteristic set
$\{\tau=\eta^{2}\}$ of the semi-classical Schrödinger flow. Therefore,
\[
\Vert\text{Op}_{h}((1-\chi)a)u\Vert_{H^{\sigma}(\mathbb{R}\times\Omega)}=O(h^{\infty})\Vert u_{0}\Vert_{L^{2}},
\]
and using a trace theorem
\[
B=\int_{0}^{T}\int_{\partial\Omega}\langle R(\partial_{n}(e^{it\Delta}u_{0})),\partial_{n}(e^{it\Delta}u_{0})\rangle d\sigma dt+O(h^{\infty})\Vert u_{0}\Vert_{L^{2}},
\]
where $R=\text{Op}(\chi a)$. Notice that a pair of reflected rays
share the same norm, therefore, by (\ref{eq:smoothq1}), the symbol
of $R$ is
\begin{gather*}
r(x_{0},t_{0},\eta,\tau)=\chi(\eta,\tau)\left(\frac{a(x_{0},\xi_{+}(\eta,\tau))-a(x_{0},\xi_{-}(\eta,\tau))}{(\xi_{+}(\eta,\tau)-\xi_{-}(\eta,\tau))\cdot n(x_{0})}\right),\\
\xi_{\pm}=\eta\pm\sqrt{\tau-\eta^{2}}n(x_{0}).
\end{gather*}
Therefore, by (\ref{eq:a_bord_ok}), (\ref{eq:smooth-a0-2}) and (\ref{eq:0-out}),
we can use the Garding inequality for the real part, the Calderon-Vaillancourt
theorem for the imaginary part in order to write
\begin{equation}
B\geq-\delta\int_{0}^{T}\int_{\partial\Omega}|\tilde{\Phi}u|^{2}d\sigma dt-c_{\text{Gard}}\Vert\chi_{b}u\Vert_{L^{2}([0,T],H^{-1/2}(\partial\Omega))}+O(h^{\infty})\Vert u_{0}\Vert_{L^{2}}.\label{eq:smot10}
\end{equation}
where $\tilde{\phi}\in C_{c}^{\infty}(\mathbb{R}^{3}\times\mathbb{R}^{3})$
is supported in $\mathcal{K}^{c}$ and $\tilde{\phi}=1$ on the support
of $\phi$, and $\chi_{b}\in C_{c}^{\infty}(\mathbb{R}^{3})$ is such
that $\chi_{b}=1$ on $\partial\Omega$. 

Moreover, by the same procedure as in \cite{MorRS}, we may suppose
that for $|x|\geq R\gg1$, $a$ is given by $a(x,\xi)=hx\cdot\xi$.
Let $\chi_{R}\in C_{c}^{\infty}$ be such that $\chi_{R}=1$ on $\left\{ |x|\leq2R\right\} $
and $\chi_{R}=0$ on $\left\{ |x|\geq3R\right\} $. We decompose 
\begin{multline*}
\int_{\Omega}\langle[i\Delta,A]u,u\rangle=\int_{\Omega}\langle[i\Delta,A]\chi_{R}u,\chi_{R}u\rangle\\
+\int_{\Omega}\langle[i\Delta,A]\chi_{R}u,(1-\chi_{R})u\rangle+\int_{\Omega}\langle[i\Delta,A](1-\chi_{R})u,\chi_{R}u\rangle\\
+\int_{\Omega}\langle[i\Delta,A](1-\chi_{R})u,(1-\chi_{R})u\rangle.
\end{multline*}
Because the commutator is truly non-negative for functions supported
in $\left\{ |x|\geq2R\right\} $, the last term is non-negative. Moreover,
the integrand of both intermediate terms are supported in $\left\{ 2R\leq|x|\leq3R\right\} $.
Therefore, taking $R$ large enough, the long-range smoothing estimate,
which is for example a consequence of the long-range resolvent estimate
of Cardoso and Vodev \cite{CardoVodev} by the procedure of \cite{MR2068304},
allows us to control them:
\begin{multline*}
\big|\int_{\mathbb{R}}\int_{\Omega}\langle[i\Delta,A]\chi_{R}u,(1-\chi_{R})u\rangle+\langle[i\Delta,A](1-\chi_{R})u,\chi_{R}u\rangle\big|\\
\lesssim\Vert\tilde{\chi}u\Vert_{L^{2}H^{1/2}}\lesssim\Vert u_{0}\Vert_{L^{2}},
\end{multline*}
where $\tilde{\chi}\in C_{c}^{\infty}$ is equal to one in $\left\{ 2R\leq|x|\leq3R\right\} $
and supported in $\left\{ |x|\geq R\right\} $. Finally, by the Garding
inequality again, using (\ref{eq:smooth-a0-2}):
\begin{equation}
\int_{0}^{T}\int_{\Omega}\langle[i\Delta,A]\chi_{R}u,\chi_{R}u\rangle\geq C\Vert\Phi u\Vert_{L^{2}H^{1/2}}-c_{\text{Gard}}\Vert\chi_{R}u\Vert_{L^{2}L^{2}},\label{smoth11}
\end{equation}
Thus, combining (\ref{eq:smooth_com}), (\ref{eq:smot10}), and (\ref{smoth11}),
using the trace theorem and controlling the lower order terms with
the estimate \textit{with logarithmic loss }we get
\begin{equation}
\Vert\Phi u\Vert_{L^{2}H^{1/2}}\leq C(\Vert u_{0}\Vert_{L^{2}}+\delta\Vert\tilde{\Phi}u\Vert_{L^{2}H^{1/2}})+C_{\delta}O(h^{\infty}).\label{eq:it}
\end{equation}

\subsection*{Iteration and conclusion}

To conclude, we would like to take $\delta>0$ small enough and iterate
(\ref{eq:it}). In order to do so, we have to take care of the potential
dependency in $\phi,\tilde{\phi},\tilde{\tilde{\phi}},\dots,\overset{\sim(k)}{\phi},\cdots$
and $\delta$ of the constants appearing in this estimate. Let us
first remark that we take all the $\overset{\sim(k)}{\phi}$ in a
given small neighborhood of the support of $\phi$ - this neighborhood
is a subset of $V$ of (\ref{eq:smooth-a0-2}). Thus, there exists
$A\geq1$ such that, for $|\alpha+\beta|\leq N$
\[
\Vert\partial_{x,\xi}^{\alpha,\beta}\overset{\sim(k)}{\phi}\Vert_{L^{\infty}}\leq A^{k}.
\]
 Therefore, the Garding constants $c_{\text{Gard}}$ in (\ref{eq:smot10}),
(\ref{smoth11}) at the $k$-th iteration can be taken as $A^{k}$.
Moreover, by (\ref{eq:smooth-a0-2}), $\xi\cdot\nabla a$ is bounded
below by a constant $C$ uniformly on the support of all the $\overset{\sim(k)}{\phi}$,
so we can choose the same constant $C$ in (\ref{smoth11}) at all
iteration. Finally, the $O(h^{\infty})$ term depends only of $\delta$.

Therefore, we can precise the constants in (\ref{eq:it}) at the $k$-th
iteration:
\[
\Vert\overset{\sim(k)}{\Phi}u\Vert_{L^{2}H^{1/2}}\leq(C+A^{k})\Vert u_{0}\Vert_{L^{2}}+C\delta\Vert\overset{\sim(k+1)}{\Phi}u\Vert_{L^{2}H^{1/2}}+C_{\delta}O(h^{\infty}),
\]
where $C$ and $A$ have no dependencies in $k$ and $\delta$ and
$C_{\delta}$ depends only of $\delta$. Thus we get
\begin{multline*}
\Vert\Phi u\Vert_{L^{2}H^{1/2}}\leq\left[C\frac{1-(C\delta)^{k+1}}{1-C\delta}+\frac{(C\delta A)-(C\delta A)^{k+1}}{1-C\delta A}\right]\Vert u_{0}\Vert_{L^{2}}\\
+(C\delta)^{k}\Vert\overset{\sim(k+1)}{\Phi}u\Vert_{L^{2}H^{1/2}}+C_{\delta}\frac{1-(C\delta)^{k+1}}{1-C\delta}O(h^{\infty})\\
\leq\left[C\frac{1-(C\delta)^{k+1}}{1-C\delta}+\frac{(C\delta A)-(C\delta A)^{k+1}}{1-C\delta A}\right]\Vert u_{0}\Vert_{L^{2}}\\
+(C\delta)^{k}\Vert\chi_{0}u\Vert_{L^{2}H^{1/2}}+C_{\delta}\frac{1-(C\delta)^{k+1}}{1-C\delta}O(h^{\infty})
\end{multline*}
where $\chi_{0}$ is compactly supported. We fix $\delta$ small enough
so that $C\delta A<1$ and let $k$ go to infinity to obtain the result.
\end{proof}
\begin{rem}
Notice that the exact same proof holds for any arbitrary domain for
which a smoothing estimate with logarithmic loss holds. Moreover,
as remarked by \cite{DatchevVasy}, we can iterate such a proof and
therefore it suffices to assume a smoothing estimate with polynomial
loss. More precisely, we initiate the argument controlling the lower
order terms by the smoothing estimate with polynomial loss, and then
iterate the proof and control the lower order terms by the previous
estimate at each step, until we reach $h^{0}$. Thus we obtain the
more general:
\end{rem}
\begin{prop}
\label{prop:dv}Let $\Omega$ be such that the following smoothing
estimate with polynomial loss holds: there exists $k>0$ such that
for all $\chi\in C_{c}^{\infty}(\mathbb{R}^{d})$ and all $u_{0}\in L^{2}$
such that $u_{0}=\psi(-h^{2}\Delta)u_{0}$, we have:
\[
\Vert\chi e^{-it\Delta_{D}}u_{0}\Vert_{L^{2}(\mathbb{R},H^{1/2}(\Omega))}\lesssim h^{-k}\Vert u_{0}\Vert_{L^{2}}.
\]
Then, a smoothing estimate without loss holds outside the trapped
set $\mathcal{K}$: that is, for all $\phi\in C_{c}^{\infty}(\mathbb{R}^{3}\times\mathbb{R}^{3})$
supported in $\mathcal{K}^{c}$, we have
\[
\Vert\text{Op}_{h}(\phi)e^{-it\Delta_{D}}u_{0}\Vert_{L^{2}(\mathbb{R},H^{1/2}(\Omega))}\lesssim\Vert u_{0}\Vert_{L^{2}}.
\]
\end{prop}

\section{Reduction to the logarithmic trapped set}

Because of Proposition \propref{smooth_logloss} and Proposition \propref{smooth_wo},
the exact same proof as in \cite{Schreodinger}, section 2, show that
the following proposition implies our main result for the Schrödinger
equation:
\begin{prop}[Strichartz estimates on a logarithmic interval near the trapped set]
\label{thm:main-1}\label{prop:semilog}There exists $\epsilon>0$
such that for all $\phi\in C_{c}^{\infty}(\mathbb{R}^{3}\times\mathbb{R}^{3})$
supported in a small enough neighborhood of $\mathcal{K}\cap\left\{ |\xi|\in[\alpha_{0},\beta_{0}]\right\} $,
we have
\begin{equation}
\Vert\text{Op}_{h}(\phi)e^{-it\Delta_{D}}\psi(-h^{2}\Delta)u_{0}\Vert_{L^{p}(0,\epsilon h|\log h|)L^{q}(\Omega)}\leq C\Vert u_{0}\Vert_{L^{2}}.\label{eq:butult}
\end{equation}
\end{prop}
Notice that, by a classical $TT^{\star}$ argument, Proposition \propref{semilog}
is a consequence of the following pointwise dispersive estimate:
\begin{equation}
\Vert Ae^{ith\Delta}\psi(-h^{2}\Delta)A^{\star}\Vert_{L^{1}\rightarrow L^{\infty}}\lesssim(ht)^{-3/2},\ \forall0\leq t\leq\epsilon|\log h|,\label{eq:pointdisp}
\end{equation}
where we denoted, here and in the sequel of this section
\[
A:=\text{Op}_{h}(\phi)
\]
in the seek of readability. 

Thus, the rest of the paper will be devoted to prove such an estimate.
The aim of this section is to show that we can reduce ourselves to
data micro-locally supported to the points that remain near the trapped
trajectories in logarithmic times. In order to do so, we first need
to generalizes some properties of the billiard flow shown in \cite{Schreodinger}: 

\subsection{Regularity of the billiard flow}

We first need the following lemma, where we denoted by $W_{\text{tan},\eta}$
an $\eta$-neighborhood of the tangent rays:
\begin{lem}
\label{lem:2cross}There exists $\eta>0$ such that any ray cannot
cross $W_{\text{tan},\eta}$ more than twice.
\end{lem}
\begin{proof}
If it is not the case, for all $n\geq0$, there exists $(x_{n},\xi_{n})\in K\times\mathcal{S}^{2}$,
where $K$ is a compact set strictly containing the obstacles, such
that $\Phi_{t}(x_{n},\xi_{n})$ cross $W_{\text{tan},\frac{1}{n}}$
at least three times. Extracting from $(x_{n},\xi_{n})$ a converging
subsequence, by continuity of the flow, letting $n$ going to infinity
we obtain a ray that is tangent to $\cup\Theta_{i}$ in at least points.
Therefore, it suffices to show that such a ray cannot exists.

Remark that, because of the non-shadows condition (\ref{eq:IK2}),
if $(x,\xi)\in W_{\text{tan}}$, if we consider the ray starting from
$(x,\xi)$ and the ray starting from $(x,-\xi)$, one of the two do
not cross any obstacle in positive times. But, if there is a ray tangent
to the obstacles in at least three points, if we consider the second
tangent point $(x_{0},\xi_{0})$, both rays starting from $(x_{0},\xi_{0})$
and $(x_{0},-\xi_{0})$ have to cross an obstacle, therefore, this
is not possible.
\end{proof}
Together with lemma 3.2 of \cite{Schreodinger}, which gives the (Hölder)
regularity of the billiard flow near tangent points for a domain with
no infinite order of contact points, we obtain, with the exact same
proof as in this previous paper - the only assumption made been which
given by \lemref{2cross}:
\begin{lem}
\label{lem:distm}Let $V$ be a bounded open set containing the convex
hull of $\cup\Theta_{i}$. Then, there exists $\mu>0$, $C>0$ and
$\tau>0$ such that, for all $x,\tilde{x}\in V$, all $\xi,\tilde{\xi}$
such that $|\xi|,|\xi'|\in[\alpha_{0},\beta_{0}]$, for all $t>0$
there exists $t'$ verifying $|t'-t|\leq$ $\tau$ such that 
\begin{equation}
d(\varPhi_{t'}(\tilde{x},\tilde{\xi}),\varPhi_{t'}(x,\xi)))\leq C^{t'}d((\tilde{x},\tilde{\xi}),(x,\xi))^{\mu}.\label{eq:div}
\end{equation}
\end{lem}
\begin{rem}
It is crucial, in the proof of this previous lemma, that a ray cannot
cross $W_{\text{tan},\eta}$ infinitely many times: indeed, regularity
is lost at each tangent point. Therefore, in the case which does not
enters the framework of \nameref{eq:IK2} of a trapped ray which is
tangent to an obstacle, this proof does not hold, and we do not know
if such a regularity of the flow is true. As this regularity is crucial
in the sequel, we think that this ``non shadow'' condition may not
be only technical, at least in the degenerated situation previously
mentioned.
\end{rem}
Finally, let us remark that
\begin{lem}
\label{lem:goesunif-1}Let $\delta>0$ and $D_{\delta}$ be a $\delta$-neighborhood
of $\text{\ensuremath{\mathcal{P}}}$. Then, for all compact $K$,
$\Phi_{t}(\rho)\longrightarrow\infty$ as $t\longrightarrow\pm\infty$
uniformly with respect to $\rho\in K\cap D_{\delta}^{c}$.
\end{lem}
\begin{proof}
It suffices to prove that the length of all trajectories in $K\cap D_{\delta}^{c}$
are uniformly bounded. If it is not the case, there exists $\rho_{n}\in D_{\delta}^{c}\cap K$
such that
\[
\text{lenght}\left\{ \Phi_{t}(\rho_{n})\right\} _{t\geq0}\cap K\longrightarrow+\infty
\]
as $n$ goes to infinity. Up to extract a subsequence, $\rho_{n}\longrightarrow\rho^{\star}\in D_{\delta}^{c}$.
Necessarily, $\text{lenght}\left\{ \Phi_{t}(\rho^{\star})\right\} _{t\geq0}\cap K=\infty$,
thus $\rho^{\star}\in\mathcal{P}$, this is not possible.
\end{proof}
\begin{lem}
\label{lem:closed}$\mathcal{K}$ is closed.
\end{lem}
\begin{proof}
Let $\rho_{n}\in\mathcal{K}$, $\rho_{n}\longrightarrow\rho$. There
exists $A>0$ such that for any $t$, $d(\pi_{x}\Phi_{t}(\rho_{n}),0_{\mathbb{R}^{3}})\leq A.$
$\pi_{x}\Phi_{t}(\cdot)$ been continuous for any fixed $t$, it suffices
to pass to the limit $n\longrightarrow\infty$ in the previous inequality
to obtain $\rho\in\mathcal{K}$. 
\end{proof}

\subsection{Reduction of the problem}

We now show that we can reduce ourselves to points that remain near
trapped trajectories in logarithmic times $T_{0}\leq t\leq\epsilon|\log h|$
in order to prove the pointwise dispersive estimate (\ref{eq:pointdisp})
in times $[T_{0},\epsilon|\log h|]$. In contrast to \cite{Schreodinger},
where we used a translation argument in the spirit of \cite{MR2672795},
we are here inspired by \cite{MR2720226}.

Let $\delta>0$. By \lemref{closed}, the projection on $\mathbb{R}^{3}\times\mathcal{S}^{2}$
of the trapped set is compact, thus there exists a finite number of
phase-space segments $(S_{k})_{1\leq k\leq N_{\delta}}$ , $S_{k}=s_{k}\times\mathbb{R}\xi_{k}\subset T^{\star}\Omega$,
$s_{i}$ been a segment of $\mathbb{R}^{3}$, such that $\mathcal{K}$
is contained in a $\delta$-neighborhood of $\cup S_{k}$. The small
quantity $\delta>0$ may be reduced a finite number of time in the
sequel. 

We will now define a microlocal partition of unity $(\Pi_{k})$. Let
$p_{k}\in C_{0}^{\infty}(T^{\star}\Omega),\ 0\leq p_{k}\leq1$ be
a family of functions such that $p_{k}$ is supported in a neighborhood
$W_{k}$ of $S_{k}$ and 
\[
\sum_{1\leq k\leq N_{\delta}}p_{k}=1\ \text{in a neighborhood of }\mathcal{K}.
\]
Let us define
\[
\Pi_{k}=\text{Op}_{h}(p_{k}),\ \forall1\leq k\leq N_{\delta}.
\]
Now, let $\chi_{0}\in C^{\infty}(\mathbb{R}^{3})$ , $0\leq\chi_{0}\leq1$
such that $\chi_{0}$ is supported sufficiently far from $\text{Con}\cup\Theta_{i}$,
and equal one far from the origin. Notice that any broken bicharacteristic
entering the support of $\chi_{0}$ from its complement remains in
it for all times. We take
\[
\Pi_{0}=\chi_{0}
\]
and let
\[
\Pi_{-1}=\text{Op}_{h}\left(1-\chi_{0}-\sum_{1\leq k\leq N_{\delta}}p_{k}\right).
\]

$\Pi_{-1}$ is defined in such a way that his symbol verifies
\[
d(\text{Supp}p_{-1},\mathcal{K})\geq d_{1}>0,
\]
therefore, by \lemref{goesunif-1}, there exists $T_{0}>0$ such that
\[
\pi_{x}\Phi_{t}(\text{Supp}p_{-1})\subset\text{Supp}\chi_{0},\ \forall|t|\geq T_{0}.
\]

Now, let $\tau>0$. It will be fixed in the sequel. In the spirit
of \cite{MR2720226}, we decompose $T=(L-1)\tau+s_{0}$, where $L\in\mathbb{N}$
and $s_{0}\in[0,\tau)$. We have
\[
e^{iTh\Delta}=e^{its_{0}\Delta}\left(e^{i\tau h\Delta}\right)^{L-1},\ e^{i\tau h\Delta}=e^{i\tau h\Delta}\sum_{-1\leq k\leq N_{\delta}}\Pi_{k}.
\]
and thus
\[
e^{iTh\Delta}=\sum_{\mathbf{k}=(k_{1},\cdots,k_{L})}e^{its_{0}\Delta}\Pi_{k_{L}}e^{i\tau h\Delta}\Pi_{k_{L-1}}\cdots\Pi_{k_{1}}e^{i\tau h\Delta},
\]
where the sum is taken over all multi-indice $\mathbf{k}\in\llbracket-1,N_{\delta}\rrbracket^{L}$.
Let us remark that, because the wavefront set of the semi-classical
Schödinger flow is invariant by the generalized bicharacteristic flow,
denoting
\[
\sigma_{\mathbf{k}}=Ae^{its_{0}\Delta}\Pi_{k_{L}}e^{i\tau h\Delta}\Pi_{k_{L-1}}\cdots\Pi_{k_{1}}e^{i\tau h\Delta}\psi(-h^{2}\Delta)A^{\star},
\]
it holds that
\begin{equation}
\rho\in WF_{h}(\sigma_{\mathbf{k}})\implies\begin{cases}
\pi_{x}\rho\in\text{Supp}\phi,\\
\Phi_{j\tau}(\rho)\in\text{Supp}q_{k_{j}} & \forall1\leq j\leq L,\\
\pi_{x}\Phi_{T}(\rho)\in\text{Supp}\phi.
\end{cases}\label{eq:WFtraj}
\end{equation}

Thus we have
\begin{lem}
Let $\mathbf{k}\in\llbracket-1,N_{\delta}\rrbracket^{L}$. If there
exists $1\leq j\leq L$ such that $k_{j}=0$ or $k_{j}=-1$, then
$\sigma_{\mathbf{k}}=O(h^{\infty})$ as an $L^{1}\rightarrow L^{\infty}$
operator.
\end{lem}
\begin{proof}
As remarked in \cite{MR2720226}, by virtue of Sobolev embeddings
it suffices to show that $\sigma_{\mathbf{k}}=O(h^{\infty})$ as an
$L^{2}\rightarrow L^{2}$ operator, thus has null operator wavefront
set. Let us suppose first that there exists $j$ such that $k_{j}=0$.
We choose $j$ to be the the first such indice. Suppose that $\rho\in WF_{h}(\sigma_{\mathbf{k}})$.
There exists $t_{0}\in[(j-1)\tau,j\tau]$ such that the spatial projection
of $\Phi_{j\tau}(\rho)$ enters the support of $\chi_{0}$ from its
complementary, thus it does not leave it. Therefore $\pi_{x}\Phi_{T}(\rho)\in\text{\text{Supp}\ensuremath{\chi}}_{\text{0}}$,
this is not possible. Thus $WF_{h}(\sigma_{\mathbf{k}})=\emptyset$.

Now, suppose that there exists $j\in[1,L-\frac{T_{\text{0}}}{\tau}]$
such that $k_{j}=-1$. Let $\rho\in WF_{h}(\sigma_{\mathbf{k}})$.
$\Phi_{j\tau}(\rho)\in\text{Supp}\Pi_{-1}$, hence
\[
\pi_{x}\Phi_{j\tau+t}(\rho)\in\text{\text{Supp}}\chi_{0},\ \forall t\geq T_{0},
\]
and we are reduced thus to the previous case. In the same way, we
exclude $j\in[\frac{T_{0}}{\tau},L]$ using the property for all $t\leq-T_{0}$.
\end{proof}
But, as the $\mathbf{k}$-sum contains at most $(N_{\delta}+2)^{\frac{\epsilon}{\tau}|\log h|}$
-- that is, a negative power of $h$ -- terms, we have
\[
\sum_{\mathbf{k}}O(h^{\infty})=O(h^{\infty}),
\]
and therefore we deduce from the previous lemma that, as an $L^{1}\rightarrow L^{\infty}$
operator
\[
Ae^{iTh\Delta}\psi(-h^{2}\Delta)A^{\star}=\sum_{\mathbf{k},k_{j}\geq1}\sigma_{\mathbf{k}}+O(h^{\infty}).
\]

Now, we will choose $\tau>0$ small enough given by the following
lemma:
\begin{lem}
\label{lem:choixtau}For all $\delta>0$, there exists  $\tau>0$
small enough so that, for every trajectory $\gamma\in\text{\ensuremath{\mathcal{P}}}$,
we have
\[
d(\rho,\gamma)<\delta,\ d(\Phi_{\tau}(\rho),\gamma)<\delta\implies\forall t\in[0,\tau],d(\Phi_{t}(\rho),\gamma)<3\delta.
\]
\end{lem}
\begin{proof}
Let $\tilde{\rho}$ realizing the distance from $\rho$ to $\gamma$.
We denote
\[
t_{0}=\inf\left\{ t\geq0,\ \text{s.t. }\pi_{x}\Phi_{t}(\rho)\in\Theta\right\} ,\ \tilde{t}_{0}=\inf\left\{ t\geq0,\ \text{s.t. }\pi_{x}\Phi_{t}(\tilde{\rho})\in\Theta\right\} .
\]
We assume that, for example, $\tilde{t}_{0}>t_{0}$. Notice that,
by the proof of \lemref{distm} from \cite{Schreodinger}, we have
\[
\forall t\in[0,\tau]\backslash(t_{0},\tilde{t}_{0}),\ d(\Phi_{t}(\rho),\Phi_{t}(\tilde{\rho}))\leq C^{\tau}\delta.
\]
Moreover, for $t\in[t_{0},\tilde{t}_{0}],$
\[
d(\Phi_{t}(\rho),\Phi_{t}(\tilde{\rho}))\leq d(\Phi_{t}(\rho),\Phi_{t_{0}}(\rho))+d(\Phi_{t_{0}}(\rho),\Phi_{t_{0}}(\tilde{\rho}))+d(\Phi_{t_{0}}(\tilde{\rho}),\Phi_{t}(\tilde{\rho})),
\]
but, as $\left\{ (\Phi_{t}(\rho)\right\} _{t\in[t_{0},\tilde{t}_{0}]}$
and $\left\{ (\Phi_{t}(\tilde{\rho})\right\} _{t\in[t_{0},\tilde{t}_{0}]}$
are straight lines 
\[
d(\Phi_{t}(\rho),\Phi_{t_{0}}(\rho))\leq|t-t_{0}||\pi_{\xi}\rho|\leq\tau\beta_{0},
\]
and similarly for $\tilde{\rho}$. Therefore
\[
d(\Phi_{t}(\rho),\Phi_{t}(\tilde{\rho}))\leq2\tau\beta_{0}+C^{\tau}\delta.
\]
We take $\tau>0$ small enough so that $2\tau\beta_{0}\leq\delta$
and $C^{\tau}\leq2$ and we get the result.
\end{proof}
The segment $S_{k_{j}}$ joins the obstacles $\Theta_{a_{j}}$ and
$\Theta_{b_{j}}$. Choosing $\delta>0$ small enough, by (\ref{eq:WFtraj}),
$\sigma_{\mathbf{k}}$ is not $O(h^{\infty})$ only if, for all $j$
\[
(a_{j}=a_{j+1}\text{ and }b_{j}=b_{j+1})\text{ or }(a_{i+1}=b_{j}).
\]
that is, only if $\gamma_{\mathbf{k}}=S_{k_{1}}\circ S_{k_{2}}\circ\cdots\circ S_{k_{L}}$
is a trajectory. Let, if it is the case, $J_{\mathbf{k}}$ be the
corresponding story of reflexions. We extract from $J_{\mathbf{k}}$
the primitive story $I_{\mathbf{k}}$, that is, $J_{\mathbf{k}}=lI_{\mathbf{k}}+r$,
$I_{\mathbf{k}}$ been primitive.

We now introduce the trapped set of an open subset in time $T$: 
\begin{defn}
Let $D$ be an open subset of $(T^{\star}\Omega\cup T^{\star}\partial\Omega)\cap\left\{ |\xi|\in[\alpha_{0},\beta_{0}]\right\} $
and $T>0$. We define the trapped set of $D$ in time $T$, denoted
$\mathcal{T}_{T}(D)$, in the following way
\[
\rho\in\mathcal{T}_{T}(D)\iff\forall t\in[0,T],\ \Phi_{T}(\rho)\in D.
\]
\end{defn}
Let us denote by $D_{I_{\mathbf{k}},\delta}$ a $\delta$-neighborhood
of $\gamma_{\mathbf{k}}\cap\left\{ |\xi|\in[\alpha_{0},\beta_{0}]\right\} $.
For $I$ a primitive story of reflexions, let $q_{I,T}\in C_{0}^{\infty}$
be such that
\begin{equation}
q_{I,T}=0\text{ outside }\mathcal{T}_{T}(D_{I,4\delta}),\ q_{I,T}=1\text{ in }\mathcal{T}_{T}(D_{I,3\delta}),\label{eq:qIT}
\end{equation}
and denote 
\[
Q_{I}^{T}:=\text{Op}_{h}(q_{I,T}).
\]
We have, by (\ref{eq:WFtraj}) and the choice of $\tau>0$ permitted
by \lemref{choixtau}
\[
\sigma_{\mathbf{k}}=\sigma_{\mathbf{k}}Q_{I_{\mathbf{k}}}^{T}+O(h^{\infty}).
\]
Now, remark that for $I$ a primitive story of reflexions
\[
Ae^{iTh\Delta}\psi(-h^{2}\Delta)A^{\star}Q_{I}^{T}=\sum_{\mathbf{k},I_{\mathbf{k}}=I}\sigma_{\mathbf{k}}Q_{I}^{T}+O(h^{\infty}),
\]
and therefore we recover
\[
\sum_{I\text{ primitive}}Ae^{iTh\Delta}\psi(-h^{2}\Delta)A^{\star}Q_{I}^{T}=Ae^{iTh\Delta}\psi(-h^{2}\Delta)A^{\star}+O(h^{\infty}).
\]

Let us finally remark that for $T\leq\epsilon|\log h|$, we have $h\leq e^{-\frac{T}{\epsilon}}$,
thus the $O(h^{\infty})$ term verifies the dispersive estimate. Therefore,
we have proven that:
\begin{lem}
\label{lem:redtrap}If the following dispersive estimate holds true
\[
\Vert\sum_{I\text{ primitive}}Ae^{iTh\Delta}\psi(-h^{2}\Delta)A^{\star}Q_{I}^{T}\Vert_{L^{1}\longrightarrow L^{\infty}}\lesssim(hT)^{-\frac{3}{2}},\ \forall T_{0}\leq T\leq\epsilon|\log h|,
\]
then the dispersive estimate (\ref{eq:pointdisp}) is true in times
$[T_{0},\epsilon|\log h|]$.
\end{lem}

\subsection{Times $0\protect\leq t\protect\leq T_{0}$ and conclusion of the
section}

Finally, notice that the construction of $Q_{I}^{T}$ does not depend
of $\phi$. We choose $\phi$ supported in a small enough neighborhood
of $\mathcal{K}$ so that, in times $0\leq t\leq T_{0}$ and for $|\xi|\in[\alpha_{0},\beta_{0}]$,
the bicharacteristic flow $\Phi_{t}(\rho)$ starting from $\rho$
has only hyperbolic points of intersection with the boundary. But,
for such points, we can use the parametrix construction of Ikawa \cite{IkawaMult,Ikawa2},
adapted to this problem in \cite{Schreodinger} and explained in the
next section in the $N$-convex framework to show that the dispersive
estimate holds true in times $0\leq t\leq T_{0}$, with a constant
depending on $T_{0}$: indeed, the flow can be writen as a finite
(depending on $T_{0}$) sum of reflected waves, each of them verifying
the dispersive estimate. 

Thus, by \lemref{redtrap}, we are reduced to show the following dispersive
estimate in order to obtain our main result, namely, we have
\begin{lem}
\label{lem:redtrap-ult}If the following dispersive estimate holds
true
\begin{equation}
\Vert\sum_{I\text{ primitive}}Ae^{iTh\Delta}\psi(-h^{2}\Delta)A^{\star}Q_{I}^{T}\Vert_{L^{1}\longrightarrow L^{\infty}}\lesssim(hT)^{-\frac{3}{2}},\ \forall T_{0}\leq T\leq\epsilon|\log h|,\label{eq:finbutult}
\end{equation}
then Strichartz estimates of Theorem \ref{th} hold true for the Schrödinger
equation.
\end{lem}
where the symbols of $Q_{I}^{T}$ were defined by (\ref{eq:qIT}).
The sequel of the paper is devoted to doing so.

Let us remark that, with the same proof as in \cite{Schreodinger},
we have, as a consequence of \lemref{distm}, 
\[
d(\mathcal{T}_{T}(\tilde{D})^{c},\mathcal{T}_{T}(D))\geq\frac{1}{4}e^{-cT}d(\tilde{D}{}^{c},D),\ \forall D\subset\tilde{D}
\]
and therefore $q_{I}^{T}$ can, and will be constructed in such a
way that, for $0\leq T\leq\epsilon|\log h|$

\begin{equation}
|\partial_{\alpha}q_{I}^{T}|\lesssim h^{-2|\alpha|c\epsilon}.\label{eq:contrq}
\end{equation}

\section{Construction of an approximate solution}

\subsection{The microlocal cut off}

We will use the reflected-phase construction of \cite{Ikawa2,IkawaMult}
and \cite{MR1254820}. It is summed up in \cite{Schreodinger}, let
us recall that $\varphi_{J}$ is the reflected phase obtained from
$\varphi$ after the story of reflexions $J$.

According to \cite{MR1254820} (remark 3.17') there exists $M>0$
such that if $J\in\mathcal{I}$, $J=rI+l$ verifies $|J|\geq M$,
and $\varphi$ verifies $(P)$, $\varphi_{J}$ can be defined in $\mathcal{U}_{I,l}^{\infty}$.
We choose $\delta>0$ small enough so that, according to the construction
of the previous section
\[
D_{I,4\delta}\subset\bigcup_{|l|\leq|I|-1}\mathcal{U}_{I,l}^{\infty},
\]
moreover, we will take $T_{0}\geq2\beta_{0}M$.

Let us recall that we are reduced to show the following dispersive
estimate:
\[
\Vert\sum_{I\text{ primitive}}Ae^{iTh\Delta}\psi(-h^{2}\Delta)A^{\star}Q_{I}^{T}\Vert_{L^{1}\longrightarrow L^{\infty}}\lesssim(hT)^{-\frac{3}{2}},\ \forall T_{0}\leq T\leq\epsilon|\log h|.
\]
For all primitive story $I$, let us define
\[
\delta_{I}^{y}(x)=\frac{1}{(2\pi h)^{3}}\int e^{-i(x-y)\cdot\xi/h}p_{I,T}(x,\xi)d\xi,
\]
where $p_{I,T}$ is the symbol associated with $P_{I}^{T}:=\psi(-h^{2}\Delta)A^{\star}Q_{I}^{T}$.
Then we have, for $u_{0}\in L^{2}$
\[
\psi(-h^{2}\Delta)A^{\star}Q_{I}^{T}u_{0}(x)=\int\delta_{I}^{y}(x)u_{0}(y)dy.
\]
Then, by linearity of the flow
\[
Ae^{ith\Delta}\psi(-h^{2}\Delta)A^{\star}Q_{I}^{T}u_{0}=\int Ae^{ith\Delta}\delta_{I}^{y}u_{0}(y)dy,
\]
and it therefore suffices to show that
\[
\sum_{I\text{ primitive}}|Ae^{iTh\Delta}\delta_{I}^{y}(x)|\lesssim(hT)^{-3/2},\ \forall T_{0}\leq T\leq\epsilon|\log h|.
\]
Finally, notice that as the operator $A$ is bounded in $L^{\infty}\rightarrow L^{\infty}$
in the same way as in \cite{Schreodinger}, it suffices only to show
that
\begin{equation}
\sum_{I\text{ primitive}}|\chi e^{iTh\Delta}\delta_{I}^{y}(x)|\lesssim(hT)^{-3/2},\ \forall T_{0}\leq T\leq\epsilon|\log h|,\label{eq:butult2}
\end{equation}
where $\chi\in C_{c}^{\infty}(\mathbb{R}^{3})$ is supported in a
neighborhood of the spatial projection of the support of $\phi$ and
equal to one on it.

In order to do so, we will construct a parametrix, that is, an approximate
solution, in time $0\leq t\leq\epsilon|\log h|$ for the semi-classical
Schrödinger equation with data $\delta_{I}^{y}$. The first step will
be to construct an approximate solution of the semi-classical Schrödinger
equation with data 
\[
e^{-i(x-y)\cdot\xi/h}p_{I,T}(x,\xi)
\]
where $\xi\in\mathbb{R}^{n}$ is fixed and considered as a parameter.
Now that we are localized around a trajectory, the construction is
exactly the same as in \cite{Schreodinger}. Let us sum it up briefly.
In the sequel of this section, $p_{I,T}$ will be denoted $p$ in
the seek of conciseness.

\subsection{Approximate solution}

We look for the solution in positives times of the equation
\[
\begin{cases}
(i\partial_{t}w-h\Delta w) & =0\ \text{in }\Omega\\
w(t=0)(x) & =e^{-i(x-y)\cdot\xi/h}p(x,\xi)\\
w_{|\partial\Omega} & =0
\end{cases}
\]
as the Neumann serie \emph{
\[
w=\sum_{J\in\mathcal{I}}(-1)^{|J|}w^{J}
\]
}where
\[
\begin{cases}
(i\partial_{t}w^{\emptyset}-h\Delta w^{\emptyset}) & =0\ \text{in }\mathbb{R}^{n}\\
w^{\emptyset}(t=0)(x) & =e^{-i(x-y)\cdot\xi/h}p(x,\xi)
\end{cases}
\]
 and, for $J\neq\emptyset$, $J=(j_{1},\cdots,j_{n})$, $J'=(j_{1},\cdots,j_{n-1})$
\begin{equation}
\begin{cases}
(i\partial_{t}w^{J}-h\Delta w^{J}) & =0\ \text{in }\mathbb{R}^{n}\backslash\Theta_{j_{n}}\\
w^{J}(t=0) & =0\\
w_{|\partial\Theta_{j_{n}}}^{J} & =w_{|\partial\Theta_{j_{n}}}^{J'}.
\end{cases}\label{eq:wJ}
\end{equation}
We will look for the $w^{J}$'s as power series in $h.$ In the sake
of conciseness, these series will be considered at a formal level
in this section, and we will introduce their expression as a finite
sum plus a reminder later, in the last section.

We look for $w^{\emptyset}$ as
\begin{gather*}
w^{\emptyset}=\sum_{k\geq0}h^{k}w_{k}^{\emptyset}e^{-i((x-y)\cdot\xi-t\xi^{2})/h},\\
w_{0}^{\emptyset}(t=0)=q(x,\xi),\ w_{k}^{\emptyset}(t=0)=0.
\end{gather*}
Solving the transport equations gives immediately
\begin{align*}
w_{0}^{\emptyset} & =p(x-2t\xi,\xi),\\
w_{k}^{\emptyset} & =-i\int_{0}^{t}\Delta w_{k-1}^{\emptyset}(x-2(s-t)\xi,s)ds\quad k\geq1.
\end{align*}

Now, starting from the phase $\varphi(x)=\frac{(x-y)\cdot\xi}{|\xi|}$,
we define the reflected phases as before and we look for $w^{J}$
as:
\begin{gather*}
w^{J}=\sum_{k\geq0}h^{k}w_{k}^{J}e^{-i(\varphi_{J}(x,\xi)|\xi|-t\xi^{2})/h},\\
w_{k}^{J}|_{t\leq0}=0,\ w_{k|\partial\Theta_{j_{n}}}^{J}=w_{k|\partial\Theta_{j_{n}}}^{J'}.
\end{gather*}
For $x\in\mathcal{U}_{J}(\varphi)$, we have 
\[
\begin{cases}
(\partial_{t}+2|\xi|\nabla\varphi_{J}\cdot\nabla+|\xi|\Delta\varphi_{J})w_{0}^{J} & =0\\
w_{0|\Theta_{j_{n}}}^{J} & =w_{0|\Theta_{j_{n}}}^{J'}\\
w_{0}^{J}|_{t\leq0} & =0
\end{cases}
\]
and
\[
\begin{cases}
(\partial_{t}+2|\xi|\nabla\varphi_{J}\cdot\nabla+|\xi|\Delta\varphi_{J})w_{k}^{J} & =-i\Delta w_{k-1}^{J}\\
w_{k|\Theta_{jn}}^{J} & =w_{k|\Theta_{j_{n}}}^{J'}\\
w_{k}^{J}|_{t\leq0} & =0.
\end{cases}
\]

Solving the transport equations along the rays by the procedure explained
in \cite{Schreodinger}, we get the exact same following expressions
of $w_{k}^{J}$ for $x\in\mathcal{U}_{J}(\varphi)$:
\begin{prop}
\label{prop:solref}We denote by $\hat{X}_{-2t}(x,|\xi|\nabla\varphi_{J})$
the backward spatial component of the flow starting from $(x,|\xi|\nabla\varphi_{J})$,
defined in the same way as $X_{-2t}(x,|\xi|\nabla\varphi_{J})$, at
the difference that we ignore the first obstacle encountered if it's
not $\Theta_{j_{n}},$ and we ignore the obstacles after $|J|$ reflections.
Moreover, for $J=(j_{1},\dots,j_{n})\in\mathcal{I}$, we denote by
\[
J(x,t,\xi)=\begin{cases}
(j_{1},\cdots,j_{k}) & \text{if }\hat{X}_{-2t}(x,|\xi|\nabla\varphi_{J})\text{ has been reflected \ensuremath{n-k} times,}\\
\emptyset & \text{if }\hat{X}_{-2t}(x,|\xi|\nabla\varphi_{J})\text{ has been reflected \ensuremath{n} times}.
\end{cases}
\]
 Then, the $w_{k}^{J}$'s are given by, for $t\geq0$ and $x\in\mathcal{U}_{J}(\varphi)$
\[
w_{0}^{J}(x,t)=\Lambda\varphi_{J}(x,\xi)p(\hat{X}_{-2t}(x,|\xi|\nabla\varphi_{J}),\xi)
\]
where
\[
\Lambda\varphi_{J}(x,\xi)=\left(\frac{G\varphi_{J}(x)}{G\varphi_{J}(X^{-1}(x,|\xi|\nabla\varphi_{J}))}\right)^{1/2}\times\cdots\times\left(\frac{G\varphi(X^{-|J|-1}(x,|\xi|\nabla\varphi_{J}))}{G\varphi(X^{-|J|}(x,|\xi|\nabla\varphi_{J}))}\right)^{1/2},
\]
and, for $k\geq1$, and $x\in\mathcal{U}_{J}(\varphi)$
\[
w_{k}^{J}(x,t)=-i\int_{0}^{t}g_{\varphi_{J}}(x,t-s,\xi)\Delta w_{k-1}^{J(x,\xi,t-s)}(\hat{X}_{-2(t-s)}(x,|\xi|\nabla\varphi_{J}),s)ds
\]
where

\[
g_{\varphi_{J}}(x,\xi,t)=\left(\frac{G\varphi_{J}(x)}{G\varphi_{J}(X^{-1}(x,|\xi|\nabla\varphi_{J}))}\right)^{1/2}\times\cdots\times\left(\frac{G\varphi_{J(x,t,\xi)}(X^{-|J(x,t,\xi)|-1}(x,|\xi|\nabla\varphi_{J}))}{G\varphi_{J(x,t,\xi)}(\hat{X}_{-2t}(x,|\xi|\nabla\varphi_{J}))}\right)^{1/2}.
\]
\end{prop}
And, by the same proof again as in \cite{Schreodinger} it implies
in particular the following three results. The first of them is about
the support of the solutions:
\begin{lem}
\label{lem:supprought}For $x\in\mathcal{U}_{J}(\varphi)$
\begin{equation}
w_{k}^{J}(x,t)\neq0\implies(\hat{X}_{-2t}(x,|\xi|\nabla\varphi_{J}),\xi)\in\text{Supp}p.\label{eq:UinfJ}
\end{equation}
And moreover
\begin{equation}
\text{Supp}w_{k}^{J}\subset\left\{ J(x,\xi,t)=\emptyset\right\} .\label{eq:suppJ}
\end{equation}
\end{lem}
It implies that we can extend it by zero outside the domains of definition
of the phases:
\begin{prop}
\label{prop:solref2}For $x\notin\mathcal{U}_{J}(\varphi)$ and $0\leq t\leq T$
we have $w_{k}^{J}(x,t)=0$.
\end{prop}
And that the have $|J|\approx t$:
\begin{lem}
\label{lem:support}There exists $c_{1},c_{2}>0$ such that for every
$J\in\mathcal{I}$, the support of $w_{k}^{J}$ is included in $\left\{ c_{1}|J|\leq t\right\} $
and which of $\chi w_{k}^{J}$ is included in $\left\{ c_{1}|J|\leq t\leq c_{2}(|J|+1)\right\} $.
\end{lem}
Now, let us recall that $q=q_{I,T}$ where $I$ is a given primitive
trajectory. We have:
\begin{lem}
If $J$ is not of the form $rI+l$, then $w_{k}^{J}=0$ for $0\leq t\leq\epsilon|\log h|$.
\end{lem}
\begin{proof}
If $w_{k}^{J}(x,\xi)\neq0$, it follows from \lemref{supprought}
that there exists a broken ray joining $(x,|\xi|\nabla\varphi_{J})$
and a point of the support of $p_{I,T}$ in time $t$ following the
complete story of reflexions $J$. By definition of the trapped set
and because $\text{Supp}p\subset\mathcal{T}_{T}(D_{I,4\delta})$,
this broken ray remains in a neighborhood of the trajectory $\gamma$
corresponding to $I$, thus $J$ can only be of the form $rI+l$. 
\end{proof}
Finally, let us notice that

\begin{lem}
\label{lem:supportchi}In times $0\leq t\leq T$, for $J=rI+l$, $\chi w_{k}^{J}$
is supported in $\mathcal{U}_{I,l}^{\infty}$.
\end{lem}
\begin{proof}
From (\ref{eq:UinfJ}), the support of $w_{k}^{J}$ consists of the
support of $q(.,\xi)$, transported along the billiard flow with initial
direction $\xi$ along the story of reflexion $J$ and then ignoring
the obstacles. Because of the non-shadow condition (\ref{eq:IK2}),
the part ignoring the obstacles is cut off by $\chi$, thus we obtain
the result.
\end{proof}

\subsection{The $\xi$ derivatives}

The following results about the directional derivatives of the phase
and the solution has been proven in \cite{Schreodinger}, where the
proof does not involve the particular two obstacles geometry. The
first one involves the critical points of the phase and its non-degeneracy:
\begin{lem}
\label{lem:nondeg}Let $J\in\mathcal{I}$ and $\mathcal{S}_{J}(x,t,\xi):=\varphi_{J}(x,\xi)|\xi|-t\xi^{2}$.
For all $t>0$ and there exists at most one $s_{J}(x,t)$ such that
$D_{\xi}\mathcal{S}_{J}(x,t,s_{J}(x,t))=0$. Moreover, for all $t_{0}>0$,
there exists $c(t_{0})>0$ such that, for all $t\geq t_{0}$ and all
$J\in\mathcal{I}$
\begin{equation}
w^{J}(x,t,\xi)\neq0\implies|\det D_{\xi}^{2}\mathcal{S}_{J}(x,t,\xi)|\geq c(t_{0})>0.\label{eq:detpo}
\end{equation}
\end{lem}
The last two permits to control the directional derivatives of the
solutions:
\begin{prop}
\label{prop:decder}For all multi-indices $\alpha,\beta$ there exists
a constant $D_{\alpha,\beta}>0$ such that the following estimate
holds on $\mathcal{U}_{I,l}^{\infty}$:
\[
|D_{\xi}^{\alpha}D_{x}^{\beta}\nabla\varphi_{J}|\leq D_{\alpha,\beta}^{|J|}.
\]
\end{prop}
\begin{cor}
\label{cor:boundsdirect}We following bounds hold on $\mathcal{U}_{I,l}^{\infty}$
\[
|D_{\xi}^{\alpha}w_{k}^{J}|\lesssim C_{\alpha}^{|J|}h^{-(2k+|\alpha|)c\epsilon}.
\]
\end{cor}

\subsection{Decay of the reflected solutions}

The principal result which permits us to estimate the decay of the
reflected solutions is the convergence of the product of the Gaussian
curvatures $\Lambda\varphi_{J}$ obtained by \cite{Ikawa2,IkawaMult}
and \cite{plaques}. It writes, in this setting
\begin{prop}
\label{prop:convL}Let $0<\lambda_{I}<1$ be the square-root of the
product of the two eigenvalues lesser than one of the Poincaré map
associated with the periodic trajectory $I$. Then, there exists $0<\alpha<1$
and a $C^{\infty}$ function $a_{I,l}$ defined in $\mathcal{U}_{I,l}^{\infty}$,
such that, for all $J=rI+l$, we have
\[
\underset{\mathcal{U}_{I,l}^{\infty}}{\sup}|\Lambda\varphi_{J}-\lambda_{I}^{r}a_{I,l}|_{m}\leq C_{m}\lambda_{I}^{r}\alpha^{|J|}.
\]
\end{prop}
In the same way as in \cite{Schreodinger}, it implies in particular:
\begin{prop}
\label{prop:essbouds}If $J=rI+l$, where $I$ is a primitive trajectory
and $l\leq|I|$, then the following bounds hold on $\mathcal{U}_{I,l}^{\infty}$:
\[
|w_{k}^{J}|_{m}\leq C_{k}\lambda_{I}^{|J|}h^{-(2k+m)c\epsilon}.
\]
Moreover, on the whole space, $|w_{k}^{J}|_{m}\leq C_{k}h^{-(2k+m)c\epsilon}.$
\end{prop}

\section{Proof of the main result}

Let $K\geq0$. By the previous section, the function
\[
(x,t)\rightarrow\frac{1}{(2\pi h)^{3}}\sum_{J=rI+l}\int\sum_{k=0}^{K}h^{k}w_{k}^{J}(x,t,\xi)e^{-i(\varphi_{J}(x,\xi)|\xi|-t\xi^{2})/h}d\xi
\]
satisfies the approximate equation
\[
\partial_{t}u-ih\Delta u=-ih^{K}\frac{1}{(2\pi h)^{3}}\sum_{J=rI+l}\int\Delta w_{K-1}^{J}(x,t,\xi)e^{-i(\varphi_{J}(x,\xi)|\xi|-t\xi^{2})/h}d\xi
\]
with data $\delta_{I,T}^{y}$. Because $e^{-i(t-s)h\Delta}$ is an
$H^{m}$-isometry and by the Duhamel formula, the difference from
the actual solution $e^{-ith\Delta}\delta^{y}$ is bounded in $H^{m}$
norm by 
\[
C\times|t|\times h^{K-3}\times\sup_{t,\xi}\sum_{J=rI+l}\Vert\Delta w_{K-1}^{J}(\cdot,t,\xi)e^{-i(\varphi_{J}(\cdot,\xi)|\xi|-t\xi^{2})/h}\Vert_{H^{m}}.
\]
Therefore,
\begin{equation}
\sum_{I\text{ primitive}}e^{-ith\Delta}\delta_{I}^{y}(x)=S_{K}(x,t)+R_{K}(x,t)\label{eq:sumdelta}
\end{equation}
with
\[
S_{K}(x,t)=\frac{1}{(2\pi h)^{3}}\sum_{J\in\mathcal{I}}\int\sum_{k=0}^{K}h^{k}w_{k}^{J}(x,t,\xi)e^{-i(\varphi_{J}(x,\xi)|\xi|-t\xi^{2})/h}d\xi
\]
and, for $0\leq t\leq\epsilon|\log h|$
\begin{equation}
\Vert R_{K}(\cdot,t)\Vert_{H^{m}}\lesssim|\log h|h^{K-3}\sup_{t,\xi}\sum_{J\in\mathcal{I}}\Vert\Delta w_{K-1}^{J}(\cdot,t,\xi)e^{-i(\varphi_{J}(\cdot,\xi)|\xi|-t\xi^{2})/h}\Vert_{H^{m}},\label{eq:rK1}
\end{equation}
where $w_{k}^{J}$ is understood to be constructed from $p_{I,T}$
when $J=rI+l$.

\subsection*{The reminder}

We first deal with the reminder term $R_{K}$. Let us denote
\[
W_{K-1}^{J}(x,t)=\Delta w_{K-1}^{J}(\cdot,t,\xi)e^{-i(\varphi_{J}(\cdot,\xi)|\xi|-t\xi^{2})/h}
\]
Notice that, by construction of the $w_{k}$'s, $w_{k}^{J}$ is supported
in a set of diameter $(C+\beta_{0}t)$. Therefore, using Proposition
\ref{prop:essbouds} to control the derivatives coming from $w_{K-1}$
and the estimate
\[
|\nabla\varphi_{J}|_{m}\leq C_{m}|\nabla\varphi|_{m}
\]
from \cite{Ikawa2} to control the derivatives coming from the phase
we get:
\[
\Vert\partial^{m}W_{K-1}^{J}\Vert_{L^{2}}\lesssim C_{K}(1+\beta_{0}t)^{\frac{1}{2}}\Vert\partial^{m}W_{K-1}^{J}\Vert_{L^{\infty}}\lesssim C_{K}(1+t)^{\frac{1}{2}}h^{-m}\times h^{-(2K+m+2)c\epsilon}
\]
and thus, by (\ref{eq:rK1}) and the Sobolev embedding $H^{2}\hookrightarrow L^{\infty}$,
for $0\leq t\leq\epsilon|\log h|$
\begin{equation}
\Vert R_{K}\Vert_{L^{\infty}}\lesssim|\log h|^{\frac{3}{2}}h^{K(1-2c\epsilon)-5-4c\epsilon}|\left\{ J\in\mathcal{I},\text{ s.t }w_{K-1}^{J}\neq0\right\} |.\label{eq:rK0}
\end{equation}
Note that $w_{K-1}^{J}(t)\neq0$ implies by \lemref{support} that
$|J|\leq c_{1}t$, and $|\left\{ J\in\mathcal{I},\text{ s.t }w_{K-1}^{J}\neq0\right\} |$
is bounded by the number of elements in 
\[
\alpha_{\lceil c_{1}t\rceil}
\]
where
\[
\alpha_{k}=\left\{ \text{sequences \ensuremath{s} in \ensuremath{\llbracket1,N\rrbracket} of lenght \ensuremath{\leq}\ensuremath{k} s.t }s_{i+1}\neq s_{i}\right\} 
\]
But
\begin{lem}
The number of elements in $\alpha_{k}$ admits the bound
\[
|\alpha_{k}|\leq C_{N}N^{k}.
\]
\end{lem}
\begin{proof}
Let us denote
\[
\beta_{k}=\left\{ \text{sequences \ensuremath{s} in \ensuremath{\llbracket1,N\rrbracket} of lenght \ensuremath{k} s.t }s_{i+1}\neq s_{i}\right\} .
\]
We have
\[
|\beta_{1}|=N
\]
and 
\[
|\beta_{k+1}|=(N-1)|\beta_{k}|.
\]
Therefore
\[
|\beta_{k}|=N(N-1)^{k-1},\ |\alpha_{k}|=\sum_{i=1}^{k}\beta_{i}+1=N\frac{(N-1)^{k}-1}{N-2}+1,
\]
and the bound holds.
\end{proof}
Thus
\begin{equation}
|\left\{ J\in\mathcal{I},\text{ s.t }w_{K-1}^{J}\neq0\right\} |\lesssim N^{t}\label{eq:wkJnonnul}
\end{equation}
and therefore, according to (\ref{eq:rK0}), for $0\leq t\leq\epsilon|\log h|$
\begin{align*}
\Vert R_{K}\Vert_{L^{\infty}} & \lesssim C_{K}|\log h|^{\frac{3}{2}}h^{K(1-2c\epsilon)-5-4c\epsilon}h^{-\epsilon\log N}\\
 & \lesssim C_{K}h^{K(1-2c\epsilon)-6-4c\epsilon-\epsilon\log N}.
\end{align*}
We take $\epsilon>0$ small enough so that $2c\epsilon\leq\frac{1}{2}$
and $\epsilon\log N\leq1$ in order to get
\[
\Vert R_{K}\Vert_{L^{\infty}}\leq C_{K}h^{\frac{K}{2}-8}.
\]
Let us fix $K=15$. Then, $\Vert R_{K}\Vert_{L^{\infty}}\leq C_{K}h^{-\frac{1}{2}}$.
Therefore, as $t\leq\epsilon|\log h|$ implies $h\leq e^{-\frac{t}{\epsilon}}$,
we get
\begin{equation}
\Vert R_{K}\Vert_{L^{\infty}}\leq C_{K}h^{-\frac{3}{2}}e^{-\frac{t}{\epsilon}}\label{eq:RK}
\end{equation}
for $0\leq t\leq\epsilon|\log h|$.

\subsection*{Times $t\protect\geq t_{0}>0$}

Let us now deal with the approximate solution $S_{K}$, $K$ been
fixed and $x$ in $\text{Supp}\chi$. Let $t_{0}>0$ to be chosen
later. For $t\geq t_{0}$, by \lemref{nondeg} we can perform a stationary
phase on each term of the $J$ sum, up to order $h$. We obtain, for
$t\geq t_{0}$
\begin{multline}
S_{K}(x,t)=\frac{1}{(2\pi h)^{3/2}}\sum_{J\in\mathcal{I}}e^{-i(\varphi_{J}(x,s_{J}(t,x))|s_{J}(t,x)|-ts_{J}(t,x)^{2})/h}\left(w_{0}^{J}(t,x,s_{J}(t,x))+h\tilde{w}_{1}^{J}(t,x)\right)\\
+\frac{1}{h^{3/2}}\sum_{J\in\mathcal{I}}R_{\text{st.ph.}}^{J}(x,t)+\frac{1}{(2\pi h)^{3}}\sum_{J\in\mathcal{I}}\int\sum_{k=2}^{K}h^{k}w_{k}^{J}(x,t,\xi)e^{-i(\varphi_{J}(x,\xi)|\xi|-t\xi^{2})/h}d\xi\label{eq:sk}
\end{multline}
where $s_{J}(t,x)$ is an eventual unique critical point of the phase
(if it does not exist, the corresponding term is $O(h^{\infty})$
and by (\ref{eq:wkJnonnul}) it does not contribute). The term $\tilde{w}_{1}^{J}$
is a linear combination of
\begin{gather*}
D_{\xi}^{2}w_{0}^{J}(t,x,s_{J}(t,x)),w_{1}^{J}(t,x,s_{J}(t,x)),
\end{gather*}
and $R_{\text{st.ph.}}^{J}$ is the reminder involved in the stationary
phase, who verifies (see for example to \cite{semibook}, Theorem
3.15)
\begin{equation}
|R_{\text{st.ph.}}^{J}(x,t)|\leq h^{2}\sum_{|\alpha|\leq7}\sup|D_{\xi}^{\alpha}w_{k}^{J}(x,\cdot,t)|.\label{eq:remstph}
\end{equation}
We recall that by \lemref{supportchi}, for $0\leq t\leq\epsilon|\log h|$,
$\chi w_{k}^{J}$ is supported in $\mathcal{U}_{I,l}^{\infty}$. Therefore,
for $0\leq t\leq\epsilon|\log h|$ and all $0\leq k\leq K-1$, we
have, if $x\in\text{Supp}\chi$, using the estimate of Proposition
\propref{essbouds}, because $w_{k}^{J}(x,\xi,\cdot)$ is supported
in $\{c_{1}|J|\leq t\leq c_{2}(|J|+1)\}$ by \lemref{support},
\[
\sum_{J\in\mathcal{I}}|w_{k}^{J}|\leq C_{k}h^{-2kc\epsilon}\sum_{\substack{J=rI+s\ |\ w_{k}^{J}\neq0\\
I\text{ primitive, }|s|\leq|I|-1
}
}\lambda_{I}^{|J|}.
\]
Thus
\[
\sum_{J\in\mathcal{I}}|w_{k}^{J}|\leq C_{k}h^{-2kc\epsilon}\sum_{I\text{ primitive}}\sum_{\substack{r\geq0\\
0\leq s\leq|I|-1
}
}\lambda_{I}^{\rho_{k}(I)+r}\lambda_{I}^{s},
\]
where we denoted
\[
\rho_{k}(I)=\inf\left\{ r\geq1\text{ s.t. }\exists s,\ w_{k}^{rI+s}\neq0\right\} ,
\]
and we get
\begin{equation}
\sum_{J\in\mathcal{I}}|w_{k}^{J}|\leq C_{k}h^{-2kc\epsilon}\sum_{\substack{I\text{ primitive}\\
\rho_{k}(I)\neq\infty
}
}\frac{1}{1-\lambda_{I}}\lambda_{I}^{\rho_{k}(I)}|I|.\label{eq:wkmaj1}
\end{equation}
Moreover, as
\begin{equation}
\rho_{k}(I)\lesssim\frac{t}{|I|}\label{eq:rhokbound}
\end{equation}
and, because as remarked in \cite{MR1254820}, if $\gamma$ is the
trajectory associated to $I$
\begin{equation}
\frac{d_{\gamma}}{\text{diam}\mathcal{C}}\leq\text{card}\gamma=|I|\leq\frac{d_{\gamma}}{d_{\text{min}}}\label{eq:cardgamma}
\end{equation}
where $\mathcal{C}$ is the convex hull of $\cup\Theta_{i}$. Therefore,
combining (\ref{eq:wkmaj1}) with (\ref{eq:rhokbound}) and (\ref{eq:cardgamma})
\begin{equation}
\sum_{J\in\mathcal{I}}|w_{k}^{J}|\lesssim C_{k}h^{-2kc\epsilon}\sum_{\substack{\mathcal{\gamma}\text{ primitive}}
}d_{\gamma}\lambda_{\gamma}^{D_{k}\frac{t}{d_{\gamma}}}.\label{eq:wkmaj2}
\end{equation}
But, by Ikawa condition (\ref{eq:IK1}), there exists $\alpha>0$
such that 
\[
\sum_{\gamma\text{ primitive}}d_{\gamma}\lambda_{\gamma}e^{\alpha d_{\gamma}}<\infty.
\]
Let us denote
\[
C_{\gamma}=\lambda_{\gamma}e^{\alpha d_{\gamma}}.
\]
Notice that, because $d_{\gamma}$ is bounded from below by $d_{\text{\text{min }}}$
uniformly with respect to $\gamma$, we have a fortiori
\[
\sum C_{\gamma}<\infty.
\]
Therefore, all $C_{\gamma}$ but a finite number are lesser than one.
Reducing $\alpha$ if necessary and taking it small enough, we can
thus assume that 
\[
0\leq C_{\gamma}\leq1,\ \forall\gamma.
\]
Hence, for $t\geq\frac{d_{\text{min}}}{D_{k}}$ we have 
\[
C_{\gamma}^{D\frac{t}{d_{\gamma}}}\leq C_{\gamma},
\]
thus, by (\ref{eq:wxp1}), for $t\geq\frac{d_{\text{min}}}{D_{k}}$
\begin{multline*}
\sum_{J\in\mathcal{I}}|w_{k}^{J}|\lesssim C_{k}h^{-2kc\epsilon}\sum_{\substack{\mathcal{\gamma}\text{ primitive}}
}d_{\gamma}\left(C_{\gamma}e^{-\alpha d_{\gamma}}\right)^{D_{k}\frac{t}{d_{\gamma}}}\\
\lesssim C_{k}h^{-2kc\epsilon}\sum_{\gamma\text{ primitive}}d_{\gamma}C_{\gamma}^{D\frac{t}{d_{\gamma}}}e^{-\alpha D_{k}t}\leq C_{k}h^{-2kc\epsilon}e^{-\alpha D_{k}t}\sum_{\gamma\text{ primitive}}d_{\gamma}C_{\gamma},
\end{multline*}
and hence, because of (\ref{eq:IK1}),
\begin{equation}
\sum_{J\in\mathcal{I}}|w_{k}^{J}|\leq C_{k}h^{-2kc\epsilon}e^{-\mu_{k}t}\text{ for }\frac{d_{\min}}{D_{k}}\leq t\leq\epsilon|\log h|.\label{eq:wxp1}
\end{equation}
for some $\mu_{k}>0$. Now, remark that for $t\leq\frac{d_{\min}}{D}$,
by (\ref{eq:wkmaj2}) we have
\[
\sum_{J\in\mathcal{I}}|w_{k}^{J}|\lesssim C_{k}h^{-2kc\epsilon}\sum_{\substack{\mathcal{\gamma}\text{ primitive}}
}d_{\gamma}\lambda_{\gamma}
\]
but because $d_{\gamma}$ are bounded below, (\ref{eq:IK1}) implies
a fortiori 
\[
\sum_{\substack{\mathcal{\gamma}\text{ primitive}}
}d_{\gamma}\lambda_{\gamma}<\infty
\]
and thus
\begin{equation}
\sum_{J\in\mathcal{I}}|w_{k}^{J}|\lesssim C_{k}h^{-2kc\epsilon}\text{ for }t_{0}\leq t\leq\frac{d_{\min}}{D_{k}}.\label{eq:wkp2}
\end{equation}
Combining (\ref{eq:wxp1}) and (\ref{eq:wkp2}) we get
\[
\sum_{J\in\mathcal{I}}|w_{k}^{J}|\leq C'_{k}h^{-2kc\epsilon}e^{-\mu_{k}t}\text{ for }t_{0}\leq t\leq\epsilon|\log h|
\]
Let us take $\epsilon>0$ small enough so that $2Kc\epsilon\leq\frac{1}{2}$.
We get, for $t_{0}\leq t\leq\epsilon|\log h|$

\begin{align}
\sum_{J\in\mathcal{I}}|w_{k}^{J}| & \leq C_{k}h^{-\frac{1}{2}}e^{-\mu t},\ 1\leq k\leq K-1,\label{eq:sk1}\\
\sum_{J\in\mathcal{I}}|w_{0}^{J}| & \lesssim e^{-\mu t}.\label{eq:sk2}
\end{align}
with 
\[
\mu=\min_{0\leq k\leq K-1}\mu_{k}>0.
\]
Moreover, using (\ref{eq:remstph}) together with (\ref{eq:wkJnonnul}),
\lemref{support} and Corollary \corref{boundsdirect} we obtain,
for $t\leq\epsilon|\log h|$
\begin{multline*}
\sum_{J\in\mathcal{I}}|R_{\text{st.ph.}}^{J}(x,t)|\leq h^{2}\sum_{J\in\mathcal{I}}\sum_{|\alpha|\leq7}\sup|D_{\xi}^{\alpha}w_{k}^{J}(x,\cdot,t)|\\
\leq h^{2-(2K+7)c\epsilon}|\left\{ J\in\mathcal{I},\text{ s.t }w_{K-1}^{J}\neq0\right\} |C^{\frac{t}{c_{1}}}\lesssim h^{2-(2K+7)c\epsilon}N^{t}C^{\frac{t}{c_{1}}}\\
\leq h^{2-(2K+7)c\epsilon}h^{-\eta\epsilon}
\end{multline*}
where $\eta>0$ depends only of $\alpha_{0},\beta_{0}$, and the geometry
of the obstacles. Therefore, choosing $\epsilon>0$ small enough
\begin{equation}
\sum_{J\in\mathcal{I}}|R_{\text{st.ph.}}^{J}(x,t)|\lesssim h\leq e^{-t/\epsilon}.\label{eq:remphst}
\end{equation}
for $t\leq\epsilon|\log h|$. In the same way we get, taking $\epsilon>0$
small enough and $t\leq\epsilon|\log h|$
\[
\sum_{J\in\mathcal{I}}|D_{\xi}^{2}w_{0}^{J}|\lesssim N^{t}C^{\frac{t}{c_{1}}}\lesssim h^{-1/4}
\]
and therefore
\begin{equation}
\sum_{J\in\mathcal{I}}|D_{\xi}^{2}w_{0}^{J}|\leq h^{-\frac{1}{2}}e^{-t/4\epsilon}.\label{eq:sk3}
\end{equation}
 So, combining (\ref{eq:sk1}), (\ref{eq:sk2}), (\ref{eq:remphst})
and (\ref{eq:sk3}) with (\ref{eq:sk}), we obtain, for some $\nu>0$
\begin{equation}
|\chi S_{K}(x,t)|\lesssim\frac{e^{-\nu t}}{h^{3/2}}\ \text{ for }t_{0}\leq t\leq T.\label{eq:SKgrand}
\end{equation}

\subsection*{Conclusion}

Combining the above estimate (\ref{eq:SKgrand}) with the control
of the reminder term (\ref{eq:RK}) and taking $t=T$ gives (\ref{eq:butult2})
and therefore the dispersive estimate (\ref{eq:finbutult}). By the
work of reduction of the third section and summed up in \lemref{redtrap-ult},
Theorem \ref{th} is therefore demonstrated for the Schrödinger equation.

\section{The wave equation}

In the case of the wave equation, the counterpart of the smoothing
estimate without loss outside the trapped set, namely the following
$L^{2}$- decay of the local energy
\begin{equation}
\Vert(Au,A\partial_{t}u)\Vert_{L^{2}(\mathbb{R},\dot{H}^{\gamma}\times\dot{H}^{\gamma-1})}\lesssim\Vert u_{0}\Vert_{\dot{H}^{\gamma}}+\Vert u_{1}\Vert_{\dot{H}^{\gamma}},\label{eq:dec_waves}
\end{equation}
where $A$ has micro-support disjoint from $\mathcal{K}$, is obtained
using the same commutator argument, writing in the case of the wave
equation as
\[
0=\int\int_{\mathbb{R}\times\Omega}\langle u,[\oblong,P]u\rangle+\int\int_{\mathbb{R}\times\partial\Omega}\langle Pu,\partial_{n}u\rangle,
\]
where $P$ is any pseudo-differential operator. Notice that the symbol
of $P$ at the border, as an operator acting on waves, has been derivated
in $\{\tau^{2}-\eta^{2}>0\}$ by \cite{MorRS}. Our method apply in
the exact same way as for the Schrödinger equation.

Once (\ref{eq:dec_waves}) is obtained, it follows as in \cite{Waves}
that we can reduce ourselves to prove the Strichartz estimates near
the trapped set in logarithmic times, namely
\[
\Vert\text{Op}_{h}(\phi)u\Vert_{L^{q}(\epsilon|\log h|,L^{r}(\Omega))}\lesssim\Vert u_{0}\Vert_{\dot{H}^{s}}+\Vert u_{1}\Vert_{\dot{H}^{s-1}}
\]
where $u_{0,1}=\psi(-h^{2}\Delta)u_{0,1}$ and $\phi$ is supported
in a small neighborhood of $\mathcal{K}$. In order to reduce ourselves
at points of the phase-space that remain near a periodic trajectory
in logarithmic times, the exact same cuting as in the third section
holds, at the difference that the flow is followed at constant speed
one. 

Then, the construction of an approximate solution is the same as in
\cite{Waves}, with the adaptations of the $N$-convex framework presented
in the fourth section. In particular, the results of non-degeneracy
of the phase and stationary points of \cite{Waves} hold, as their
proof does not rely on the particular two-convex geometry. Thus, we
can perform the same stationary phase argument as in \cite{Waves},
the difference with the Schrödinger equation been that the phase is
now stationary on plain lines due to the constant speed of propagation,
and we obtain the good scale in $h$. Now, the only difference with
the conclusion section of \cite{Waves} is that we cannot deal with
\[
\sum_{J\in\mathcal{I}}
\]
as in the two convex case. But we can do it in the exact same way
as presented in the fifth section, using the strong hyperbolic setting
assumption (\ref{eq:IK1}), in order to deduce the sufficient time
decay. Thus the appropriate dispersive estimate for the waves is obtained
and the theorem follows.

\subsection*{Aknowledgments}

The author is grateful to Jared Wunsch for having indicated him the
paper of Datchev and Vasy \cite{DatchevVasy} which permits to remark
that we are able to obtain the more general smoothing estimate of
Proposition \propref{dv}, and to Fabrice Planchon and Nicolas Burq
for many discussions about the problem.

\bibliographystyle{amsalpha}
\bibliography{refs,/Users/David/Desktop/Projets/arXiv_waves/megaref}

\providecommand{\bysame}{\leavevmode\hbox to3em{\hrulefill}\thinspace}
\providecommand{\MR}{\relax\ifhmode\unskip\space\fi MR }
\providecommand{\MRhref}[2]{%
  \href{http://www.ams.org/mathscinet-getitem?mr=#1}{#2}
}
\providecommand{\href}[2]{#2}
\begin{thebibliography}{HTW06}

\bibitem[BGH10]{MR2720226}
Nicolas Burq, Colin Guillarmou, and Andrew Hassell, \emph{Strichartz estimates
  without loss on manifolds with hyperbolic trapped geodesics}, Geom. Funct.
  Anal. \textbf{20} (2010), no.~3, 627--656. \MR{2720226 (2012f:58068)}

\bibitem[BGT04]{MR2068304}
N.~Burq, P.~G{\'e}rard, and N.~Tzvetkov, \emph{On nonlinear {S}chr\"odinger
  equations in exterior domains}, Ann. Inst. H. Poincar\'e Anal. Non Lin\'eaire
  \textbf{21} (2004), no.~3, 295--318. \MR{2068304 (2005g:35264)}

\bibitem[BLP08]{BLP}
Nicolas Burq, Gilles Lebeau, and Fabrice Planchon, \emph{Global existence for
  energy critical waves in 3-{D} domains}, J. Amer. Math. Soc. \textbf{21}
  (2008), no.~3, 831--845. \MR{2393429}

\bibitem[Bou11]{Bouclet}
Jean-Marc Bouclet, \emph{Strichartz estimates on asymptotically hyperbolic
  manifolds}, Anal. PDE \textbf{4} (2011), no.~1, 1--84. \MR{2783305}

\bibitem[BSS09]{BSS}
Matthew~D. Blair, Hart~F. Smith, and Christopher~D. Sogge, \emph{Strichartz
  estimates for the wave equation on manifolds with boundary}, Ann. Inst. H.
  Poincar\'e Anal. Non Lin\'eaire \textbf{26} (2009), no.~5, 1817--1829.
  \MR{2566711}

\bibitem[BT07]{BoucletTzvetkov}
Jean-Marc Bouclet and Nikolay Tzvetkov, \emph{Strichartz estimates for long
  range perturbations}, Amer. J. Math. \textbf{129} (2007), no.~6, 1565--1609.
  \MR{2369889}

\bibitem[Bur93]{MR1254820}
Nicolas Burq, \emph{Contr\^ole de l'\'equation des plaques en pr\'esence
  d'obstacles strictement convexes}, M\'em. Soc. Math. France (N.S.) (1993),
  no.~55, 126. \MR{1254820}

\bibitem[Bur03]{MR2001179}
N.~Burq, \emph{Global {S}trichartz estimates for nontrapping geometries: about
  an article by {H}. {F}.\ {S}mith and {C}. {D}.\ {S}ogge: ``{G}lobal
  {S}trichartz estimates for nontrapping perturbations of the {L}aplacian''
  [{C}omm. {P}artial {D}ifferential {E}quation {\bf 25} (2000), no. 11-12
  2171--2183; {MR}1789924 (2001j:35180)]}, Comm. Partial Differential Equations
  \textbf{28} (2003), no.~9-10, 1675--1683. \MR{2001179}

\bibitem[Bur04]{MR2066943}
\bysame, \emph{Smoothing effect for {S}chr\"odinger boundary value problems},
  Duke Math. J. \textbf{123} (2004), no.~2, 403--427. \MR{2066943
  (2006e:35026)}

\bibitem[CV02]{CardoVodev}
F.~Cardoso and G.~Vodev, \emph{Uniform estimates of the resolvent of the
  {L}aplace-{B}eltrami operator on infinite volume {R}iemannian manifolds.
  {II}}, Ann. Henri Poincar\'e \textbf{3} (2002), no.~4, 673--691. \MR{1933365}

\bibitem[DV13]{DatchevVasy}
Kiril Datchev and Andr\'as Vasy, \emph{Propagation through trapped sets and
  semiclassical resolvent estimates}, Microlocal methods in mathematical
  physics and global analysis, Trends Math., Birkh\"auser/Springer, Basel,
  2013, pp.~7--10. \MR{3307787}

\bibitem[GV85]{GV85}
J.~Ginibre and G.~Velo, \emph{Scattering theory in the energy space for a class
  of nonlinear {S}chr\"odinger equations}, J. Math. Pures Appl. (9) \textbf{64}
  (1985), no.~4, 363--401. \MR{839728}

\bibitem[HTW06]{HassellTaoWunsch}
Andrew Hassell, Terence Tao, and Jared Wunsch, \emph{Sharp {S}trichartz
  estimates on nontrapping asymptotically conic manifolds}, Amer. J. Math.
  \textbf{128} (2006), no.~4, 963--1024. \MR{2251591}

\bibitem[Ika82]{IkawaMult}
Mitsuru Ikawa, \emph{Decay of solutions of the wave equation in the exterior of
  two convex obstacles}, Osaka J. Math. \textbf{19} (1982), no.~3, 459--509.
  \MR{676233}

\bibitem[Ika88]{Ikawa2}
\bysame, \emph{Decay of solutions of the wave equation in the exterior of
  several convex bodies}, Ann. Inst. Fourier (Grenoble) \textbf{38} (1988),
  no.~2, 113--146. \MR{949013}

\bibitem[ILLP]{ILLPGeneral}
Oana Ivanovici, Richard Lascar, Gilles Lebeau, and Fabrice Planchon,
  \emph{Dispersion for the wave equation inside strictly convex domains {II}:
  the general case}, Preprint.

\bibitem[ILP14]{ILPAnnals}
Oana Ivanovici, Gilles Lebeau, and Fabrice Planchon, \emph{Dispersion for the
  wave equation inside strictly convex domains {I}: the {F}riedlander model
  case}, Ann. of Math. (2) \textbf{180} (2014), no.~1, 323--380. \MR{3194817}

\bibitem[Iva10]{MR2672795}
Oana Ivanovici, \emph{On the {S}chr\"odinger equation outside strictly convex
  obstacles}, Anal. PDE \textbf{3} (2010), no.~3, 261--293. \MR{2672795
  (2011j:58037)}

\bibitem[Iva12]{OanaCounterex}
\bysame, \emph{Counterexamples to the {S}trichartz inequalities for the wave
  equation in general domains with boundary}, J. Eur. Math. Soc. (JEMS)
  \textbf{14} (2012), no.~5, 1357--1388. \MR{2966654}

\bibitem[Kap89]{Kapitanskii}
L.~V. Kapitanski{\u\i}, \emph{Some generalizations of the
  {S}trichartz-{B}renner inequality}, Algebra i Analiz \textbf{1} (1989),
  no.~3, 127--159. \MR{1015129}

\bibitem[KT98]{KeelTao}
Markus Keel and Terence Tao, \emph{Endpoint {S}trichartz estimates}, Amer. J.
  Math. \textbf{120} (1998), no.~5, 955--980. \MR{1646048}

\bibitem[Laf17a]{Waves}
D.~Lafontaine, \emph{About the wave equation outside two strictly convex
  obstacles}, Preprint, https://arxiv.org/abs/1711.09734 (2017).

\bibitem[Laf17b]{Schreodinger}
\bysame, \emph{Strichartz estimates without loss outside two strictly convex
  obstacles}, Preprint, https://arxiv.org/abs/1709.03836 (2017).

\bibitem[LS95]{LindbladSogge}
Hans Lindblad and Christopher~D. Sogge, \emph{On existence and scattering with
  minimal regularity for semilinear wave equations}, J. Funct. Anal.
  \textbf{130} (1995), no.~2, 357--426. \MR{1335386}

\bibitem[Met04]{Metcalfe}
Jason~L. Metcalfe, \emph{Global {S}trichartz estimates for solutions to the
  wave equation exterior to a convex obstacle}, Trans. Amer. Math. Soc.
  \textbf{356} (2004), no.~12, 4839--4855. \MR{2084401}

\bibitem[MRS77]{MorRS}
Cathleen~S. Morawetz, James~V. Ralston, and Walter~A. Strauss, \emph{Decay of
  solutions of the wave equation outside nontrapping obstacles}, Comm. Pure
  Appl. Math. \textbf{30} (1977), no.~4, 447--508. \MR{0509770}

\bibitem[MSS93]{MoSeSo}
Gerd Mockenhaupt, Andreas Seeger, and Christopher~D. Sogge, \emph{Local
  smoothing of {F}ourier integral operators and {C}arleson-{S}j\"olin
  estimates}, J. Amer. Math. Soc. \textbf{6} (1993), no.~1, 65--130.
  \MR{1168960}

\bibitem[Smi98]{SmithC11}
Hart~F. Smith, \emph{A parametrix construction for wave equations with
  {$C^{1,1}$} coefficients}, Ann. Inst. Fourier (Grenoble) \textbf{48} (1998),
  no.~3, 797--835. \MR{1644105}

\bibitem[SS95]{SS95}
Hart~F. Smith and Christopher~D. Sogge, \emph{On the critical semilinear wave
  equation outside convex obstacles}, J. Amer. Math. Soc. \textbf{8} (1995),
  no.~4, 879--916. \MR{1308407}

\bibitem[SS00]{SmithSoggeNonTrapping}
\bysame, \emph{Global {S}trichartz estimates for nontrapping perturbations of
  the {L}aplacian}, Comm. Partial Differential Equations \textbf{25} (2000),
  no.~11-12, 2171--2183. \MR{1789924}

\bibitem[ST02]{StaffTata}
Gigliola Staffilani and Daniel Tataru, \emph{Strichartz estimates for a
  {S}chr\"odinger operator with nonsmooth coefficients}, Comm. Partial
  Differential Equations \textbf{27} (2002), no.~7-8, 1337--1372. \MR{1924470}

\bibitem[Str77]{Strichartz}
Robert~S. Strichartz, \emph{Restrictions of {F}ourier transforms to quadratic
  surfaces and decay of solutions of wave equations}, Duke Math. J. \textbf{44}
  (1977), no.~3, 705--714. \MR{0512086}

\bibitem[Tat02]{Tataruns}
Daniel Tataru, \emph{Strichartz estimates for second order hyperbolic operators
  with nonsmooth coefficients. {III}}, J. Amer. Math. Soc. \textbf{15} (2002),
  no.~2, 419--442. \MR{1887639}

\bibitem[Zwo12]{semibook}
Maciej Zworski, \emph{Semiclassical analysis}, Graduate Studies in Mathematics,
  vol. 138, American Mathematical Society, Providence, RI, 2012. \MR{2952218}

\end{thebibliography}

\end{document}